\documentclass[10pt,twoside,reqno]{amsart}
\usepackage{amsmath, amsthm, amscd, amsfonts, amssymb, graphicx, color}
\usepackage[bookmarksnumbered, colorlinks, plainpages]{hyperref}
\usepackage{mathrsfs}
\usepackage{algorithm2e}
\usepackage{hhline,booktabs}
\usepackage{multirow}
\usepackage{float}
\usepackage{placeins}
\usepackage{subfigure}
\usepackage{lipsum}
\usepackage{upgreek}
\usepackage{hhline,booktabs}
\usepackage{multirow}
\usepackage{subfigure}
\usepackage[mathscr]{euscript}
\usepackage{multirow,bigstrut,rotating,float}
 \usepackage{algorithmic}
 \usepackage{kantlipsum}

\numberwithin{equation}{section}
\newtheorem{theorem}{Theorem}[section]
\newtheorem{corollary}[theorem]{Corollary}
\newtheorem{lemma}[theorem]{Lemma}

\newtheorem{definition}[theorem]{Definition}

\begin{document}

\title{On the permanent of random tensors}

\thanks{{\scriptsize
$^\ast$Corresponding author}}
\maketitle

\begin{center}

	\textbf{\textbf{Malihe Nobakht-Kooshkghazi}}$^{\ast,\S}$,  \textbf{\textbf{Hamidreza Afshin}}$^\S$ \\ [0.2cm]

	$^\S${\small \textit{Department of Mathematics, Vali-e-Asr University of Rafsanjan,\\Rafsanjan, Iran}}\\
	\texttt{e-mails: m.nobakht88@gmail.com,afshin@vru.ac.ir}
\end{center}


\begin{abstract}
The exact computation of permanent for high-dimensional tensors is a hard problem. Having in mind the applications of permanents in other fields, providing an algorithm for the approximation of tensor permanents is an attractive subject. In this paper, we design a deterministic quasi-polynomial time algorithm and a PTAS that computes the permanent of complex random tensors that its module of the mean is at least $\frac{1}{polylog(n)}$.

\end{abstract}

\pagestyle{myheadings}
\markboth{\rightline {\scriptsize   
}}
         {\leftline{\scriptsize    }}

\bigskip


\section{Introduction}
A \textit{tensor} $ A=(a_{i_{1}\dots i_{d}})_{n_{1}\times \dots \times n_{d}} $ is a multi-array of entries $ a_{i_{1} \dots i_{d}} \in \Bbb F$, where $ i_{j} = 1,\dots,n_{j} $ for $ j=1,\dots,d $ and $ \Bbb F $ is a field. In this paper, we consider complex tensors, that is, those tensors for which $\Bbb F = \Bbb C $. Here, $ d $ and $ (n_{1},\dots,n_{d}) $ are referred to as the \textit{order} and the \textit{dimension} of $A$, respectively. When $ n_{1}=n_{2}=\cdots=n_{d}=n $, we say that $A$ is a tensor of order $ d $ and dimension $ n $. In this paper, we only consider tensors of order $ d $ and dimension $ n $. We refer the interested reader to \cite{ten} for more information on the theory of tensors.

 The permanent of a matrix was introduced by Cauchy \cite{cau} in 1812. We know that the problem of computing the permanent of a non-negative matrix, or even that of a $(0,1)$-matrix, lies in the class of $\#$P-hard problems \cite{Var179}. Ryser presented an algorithm for computing the permanent of an $ n\times n $ matrix that runs in $ n2^{n} $ time \cite{rys}. In 2017, Alexander Barvinok presented a quasi-polynomial deterministic algorithm to approximate the permanent of an $n\times n$ matrix $A = (a_{ij})$ with $ \delta \leq a_{ij}\leq 1 $, for all $ i,j $, in $n^{\mathcal{O}(\ln n-\ln \varepsilon)}$ time \cite{all1}. 
 
In physics, the permanent of a matrix appears as the amplitudes of bosonic systems. Motivated by these connections, in 2011, Aaronson and Arkhipov introduced a model of quantum computation known as Boson Sampling \cite{Aa}. Even though boson sampling is a highly rudimentary quantum computational device, Aaronson and Arkhipov showed that assuming plausible conjectures about hardness of approximation and anti-concentration of the permanent of i.i.d. Gaussian matrices, approximately sampling from the Boson Sampling device is hard for a classical computer unless notable conjectures in computer science (namely non-collapse of the polynomial hierarchy) do not hold. To be more specific, one of the conjectures known as the hardness of approximating the permanent states that it is $\#$P-hard to approximate the permanent of matrix with i.i.d. complex Gaussian entries to within inverse-polynomial multiplicative error. In \cite{eld}, Eldar and Mehraban showed that if we sample matrices from complex Gaussian entries with nonzero (inverse polyloglog) but vanishing mean, then there is a quasi-polynomial time algorithm to approximate the permanent within multiplicative error. The technique Eldar and Mehraban used is based on a recent approach by Alexander Barvinok that considers a low-degree polynomial "additive" approximation to the logarithm of the permanent polynomial (this gives a multiplicative approximation to the permanent) \cite{birv2}. In particular,  Eldar and Mehraban considered the single variable polynomial $g(z) = per (J + A z)$, where $J$ is the all-ones matrix. They show that the disk of radius $L$ in the complex plan has at most $\mathcal{O}(L^{2})$ zeros of $g$, and hence, there is a curve that connects $z = 0$ to $z = L$ that avoids all the roots with a significant margin. They then use analytic continuation for $\log g(z)$ along this curve and showed that low-order Taylor polynomial expansion along this curve converges. In 2021, Zhengfeng Ji, Zhihan Jin and Pinyan Lu subsequently improved this result by improving the value of the mean to inverse polylogarithmic \cite{pin}. They furthermore simplified the description of the algorithm. Ji et al. used nontrivial probabilistic tools to showed that if we simply Taylor expand $g(z)$ in powers of $z$, then the first few ($\log n$) terms of the expansion would constitute a good multiplicative approximation.

 
  The permanent of a tensor was introduced by Cayley \cite{cay} in 1849. In \cite{dg}, the \textit{permanent} of tensor $A$ of order $d$ and dimension $n_{1} \times \dots \times n_{d}$ is defined by
\begin{equation*}
	per(A):=\sum\limits_{\sigma_{k}}\prod_{i=1}^{n_{1}} a_{i\sigma_{2}(i) \dots \sigma_{d}(i)}, 
\end{equation*}
where the summation runs over all one-to-one functions $ \sigma_{k} $ from $ \lbrace 1, \dots , n_{1} \rbrace $ to $ \lbrace 1, \dots , n_{k} \rbrace $ and $k=2, \dots ,d$, with $per(A)=0 $ if $n_{1} > n_{k}$ for some $k$ .
 
 In recent years, some properties and applications of the permanents of tensors have been studied by Dow and Gibson \cite{dg}, and also by Wang and Zhang \cite{per}. The use of permanents of tensors facilitates the solutions of many problems in other branches of mathematics and other sciences. For example, transversals in Latin hypercubes, $1$-factors of hypergraphs, MDS codes, combinatorial designs, the number of Steiner systems, and the description of the interactions of elementary particles in physics \cite{tar2016} all depend on the permanents of tensors. 
 
 Since the computation of the permanent of non-negative matrices is $\#P$-complete and tensors are a generalization of matrices to higher dimensions, computing the permanent of non-negative tensors is harder than matrices. Because of this hardness in the worst case, research has also focused on computing an approximation for the value of permanent. Similar to the matrix case, the computational complexity of computing the permanent of a tensor is important to complexity theory and has been studied.
  In \cite{alex}, an algorithm is presented for computing the permanent of tensors whose entries of it are in the special form. 
   In 2019, Alexander Barvinok presented a quasi-polynomial time algorithm to approximate the permanent of an $n\times n$ complex matrix, where the absolute value of each diagonal entry is at least $\lambda$ times bigger than the sum of the absolute values of all other entries in the same row, for any $\lambda > 1$. Also, they generalized this result to the permanent of tensors \cite{birv2}. Their method presented an algorithm that computes an additive approximation of $ \ln(per(A)) $ and then exponentiates the result. 
  
   The approach in \cite{pin} stated here for the permanent of a random tensor that module of its mean value is at least $\frac{1}{polylog(n)}$ ($\vert \mu \vert > \frac{1}{polylog(n)}$), where $d\geq 3 $ and $ d \in \mathcal{O}(polylog(n)) $. Here instead of trying to compute $ per(X) $ directly, we write $\frac{per(X)}{(n!)^{d-1}} $ as $ \sum\limits_{k=0}^{t} a_{k}z^{k}$, where $ z=\frac{1}{\mu} $ and $ t=\ln n+\ln \frac{1}{\varepsilon}$ and $ \varepsilon $ is the approximation parameter, then we give a simple approximation for this summation. Also, we show that this approximation is a good approximation by taking a detour by artificially defining several related quantities and using the moment bounds and recursions among them to inductively show the desired approximation rate. Also, our algorithm has time complexity $ n^{\mathcal{O}(d \ln \frac{n}{\varepsilon})}$. The proof strategies here are similar to those in the matrix case. 

  
 The remainder of this paper is organized as follows. In Section $2$, we study the relation between the permanents of some tensors and finite series. In section $ 3 $, we give a summary of the main theorem and general structure of the proof of the main theorem. In Sections $4,5$ and $ 6 $, we prove the technical lemmas that will use in section $ 7 $. In Section $7$, we present a deterministic quasi-polynomial time algorithm which approximates the permanent of the tensor $R$ sampled from $\mathscr{T}_{d,n,\mu}$ with $\vert \mu \vert  \geq (\frac{\ln n}{(d-1)^{6}})^{-c} $, where $d \geq 3$ and $ d \in \mathcal{O}(polylog(n)) $. Also, we modify it to a PTAS. We conclude the paper with a brief conclusion in Section $8$.


\section{Preliminary}\label{sec:6}

\noindent 
In this section, first, we present some definitions that will be used in the proofs of our theorems. Next, we present a relation between the permanents of some tensors and finite series. 

In this paper, we denote by $T_{d,n}$ the set of all tensors of order $d$ and dimension $n$. We use $[n]$ for the set $\lbrace 1,2,\dots,n\rbrace$. Also, the sets $\mathbb{N}$, $\mathbb{R}$ and $\mathbb{C}$ are the sets of natural numbers, real numbers and complex numbers, respectively. We define $n^{\underline{k}}:= n (n-1) (n-2) \dots (n-k+1)$ to denote the downward factorial. The sets $C_{n,k}$ and $P_{n,k}$ denote the sets of all $k$-combinations and $k$-permutations of $[n]$, respectively. Also, $\delta_{i_{1},i_{2},\dots, i_{d}}$ is the Kronecker function, that is, $\delta_{i_{1},i_{2},\dots ,i_{d}}=1$ if $i_{1}=i_{2}=\dots= i_{d} $, and $\delta_{i_{1}i_{2}\dots i_{d}}=0$ otherwise.

\begin{definition}[\cite{dg}]
Let $A=(a_{i_{1}\dots i_{d}}) $ be a tensor of order $d$ and dimension $n$ and $ k \in \lbrace 0,1,\dots,d \rbrace $, a $ k$\textit{-dimensional plane} in $A$ is a subtensor obtained by fixing $ d-k $ indices and letting the other $k$ indices vary from $1$ to $n$. A $1$-dimensional plane is said to be a \textit{line}, and a $ (d-1) $-dimensional plane is a \textit{hyperplane}. 
\end{definition}
\begin{definition}[\cite{dg}]
For a tensor $ A = (a_{i_{1}\dots i_{d}})_{n_{1} \times \dots \times n_{d}}$, the hyperplanes obtained by fixing $ i_{k}, 1 \leq k \leq d $, are said to be \textit{hyperplanes of type} $k$.
\end{definition}

\begin{definition}[\cite{tar2016}]
	A pair $ H=(X,W) $ is called a \textit{hypergraph} with \textit{vertex set} $X$ and \textit{hyperedge set} $W$, where each hyperedge $ w \in W $ is a subset of the vertices in $X$. A hypergraph $H$ is called $k$\textit{-uniform} if each of its hyperedges consists of $k$ vertices.
\end{definition}
\begin{definition}[\cite{dg}]
	Let $A=(a_{i_{1}\dots i_{d}}) $ be a tensor of order $d$ and dimension $n_{1} \times \dots \times n_{d}$. The generalization of the bipartite graph of a matrix to the tensor is the $d$-partite hypergraph $ H(A) = (V,W) $ with the vertex set 
	\[ V=\lbrace v_{j}^{k}: k=1,\dots ,d , j=1,2,\dots ,n_{k} \rbrace, \]
	and the edge set
	\[W = \lbrace (v_{i_{1}}^{1}, v_{i_{2}}^{2},\dots ,v_{i_{d}}^{d}): a_{i_{1}i_{2}\dots i_{d}}\neq 0) \rbrace. \]
\end{definition}

\begin{definition}[\cite{pin}]
Suppose that $x_{1},x_{2},\dots,x_{n} \in \mathbb{C}$ and $0\leq k\leq n$. The \textit{power sum} is defined by
\[
S_{k}(x_{1},x_{2},\dots,x_{n}):=\sum\limits_{i=1}^{n} x_{i}^{k}, 
\]
and the \textit{$k$th elementary symmetric polynomial} is defined by 
\[
e_{k}(x_{1},x_{2},\dots,x_{n}):=\sum\limits_{\lbrace i_{1},i_{2},\dots,i_{k}\rbrace \in C_{n,k}} x_{i_{1}}x_{i_{2}}\dots x_{i_{k}},
\]
with the convention that $e_{0}(x_{1},x_{2},\dots,x_{n})=1$. We will write $S_{k}(n)$ and $e_{k}(n)$ if the variables $x_{i}$ are clear from the context.
\end{definition}
The following definition defines families of distributions with the properties stated. 

\begin{definition}[\cite{pin}]
The \textit{entry distribution $\mathcal{D}_{\mu} $} with mean value $\mu \in \mathbb{C}$ is a distribution over complex numbers such that
\[
E_{x \sim \mathcal{D}_{\mu}}[x]=\mu, \qquad Var_{x \sim D_{\mu}}[x]=1,
\]
and 
\[
E_{x \sim \mathcal{D}_{\mu}}[\vert x-\mu \vert^{3}]=\rho< \infty.
\]
We use $\mathcal{D}$ to denote $\mathcal{D}_{0}$.
\end{definition}
In this paper, we use $\xi$ to denote the quasi-variance of $\mathcal{D}_{\mu}$:
\[
\xi=E_{x \sim \mathcal{D}_{\mu}}[(x-\mu)^{2}].
\]
We see that
\begin{eqnarray}\label{m1}
\nonumber  \vert \xi \vert &=& \vert E_{x \sim \mathcal{D}_{\mu}}[(x-\mu)^{2}]\vert\\
\nonumber &\leq & E_{x \sim \mathcal{D}_{\mu}}[\vert x-\mu \vert^{2}]\\
&=& Var_{x \sim \mathcal{D}_{\mu}}(x)=1.
\end{eqnarray}
In what follows, we generalize the definition of matrix distribution \cite{pin} to tensors.
\begin{definition}
The tensor distribution $\mathscr{T}_{d,n,\mu}$ is the distribution over $A \in T_{d,n}(\mathbb{C})$ such that the entries of $A$ are $i.i.d.$ sampled from $\mathcal{D}_{\mu}$.
\end{definition}
\begin{definition}
Let $A \in T_{d,n}(\mathbb{C}), B \in T_{d,k}(\mathbb{C})$. We write $B \subseteq_{k} A$ if $B$ is a $k\times k\times \dots \times k$ subtensor of $A$.   
\end{definition} 
In the following lemma, we represent $per(J+zA)$ by a finite sum.
\begin{lemma}\label{11}
Let $A \in T_{d,n}(\mathbb{C})$ be an arbitrary tensor, and $J\in T_{d,n}$ be the tensor whose all entries are equal to $1$. For $k=0,1,\dots,n$ define
 \begin{eqnarray}\label{rs1}
a_{k}=\dfrac{1}{(n^{\underline{k}})^{d-1}} \sum\limits_{B\subseteq_{k} A}per(B).
\end{eqnarray}
Then for every $z \in \mathbb{C}$,
\begin{eqnarray*}
\dfrac{per(J+zA)}{(n!)^{d-1}}=\sum\limits_{k=0}^{n}a_{k}z^{k}.
\end{eqnarray*}
\end{lemma}
\begin{proof}
Let $M$ be any tensor of order $d$ and dimension $n$. Let $G_{M}$ be the corresponding $d$-partite hypergraph associated to $M$, where the weight of edge $e=(i_{1}i_{2}\dots i_{d})$ is simply $M_{i_{1}i_{2}\dots i_{d}}$.
Define the weight of any perfect matching in $G_{M}$ to be the product of the weights of all edges in it. We know that the permanent of $M$ is a sum of the weights of all perfect matchings of $G_{M}$. We can split any perfect matching of $G_{J+zA}$ into combinations of $k$-perfect matchings of $G_{zA}$ and an $n-k$-perfect matching of $G_{J}$ for some $k$ \cite{dg}. Regarding $per(J+zA)$ as a polynomial of $z$, the coefficient of $z^{k}$ is the summation of the weights over all $k$-perfect matchings of $G_{zA}$ times the weights of perfect matchings in $G_{J}$, which is $((n-k)!)^{d-1}$.  
\end{proof}

\section{Summary of the main result}
In this section, we propose an algorithm for the permanent of a random tensor $ R\sim \mathscr{T}_{d,n,\mu} $ with $\vert \mu \vert  \geq (\frac{\ln n}{(d-1)^{6}})^{-c} $. Also, we give a brief overview of its proof.
\begin{theorem}\label{22}
For any constant $ c \in (0,\frac{1}{8}) $, there exists a deterministic quasi-polynomial time algorithm $ \mathscr{P} $, where the tensor $ R $ sampled from $ \mathscr{T}_{d,n,\mu} $ and $d \in \mathcal{O}(polylog(n))  $, for $\vert \mu \vert  \geq (\frac{\ln n}{(d-1)^{6}})^{-c} $, and a real number $\varepsilon \in (0,1)$, are considered as inputs. The algorithm computes a complex number $ \mathscr{P}(R,\varepsilon) $ that approximates the permanent of the tensor $R$ such that
\[
P\bigg( \bigg \vert 1-\dfrac{\mathscr{P}(R,\varepsilon)}{per(R)} \bigg\vert \leq \varepsilon \bigg)\geq 1-o(1),
\]
where the probability is taken over the random tensor $ R $.
\end{theorem}
 To do this, we present a tensor $ X=J+zA $, where $z$ is a complex number and $A$ is a random tensor with i.i.d entries sampled from $\mathcal{D}_{0} $, and we let $z=\frac{1}{\mu}$ at the end. It is clear that $ \frac{X}{z} \sim \mathscr{T}_{d,n,\mu}$, where $ \frac{X}{z}=\frac{x_{i_{1}i_{2}\dots i_{d}}}{z} $ for all $ i_{1},i_{2},\dots, i_{d} $, with $z=\frac{1}{\mu}$. Thus, our problem is equivalent to finding the permanent of $X$. For this, we approximate $ \frac{per(X)}{(n!)^{d-1}} $. To study this subject, similar to matrix case in \cite{pin}, we use some parameters that satisfy the following relations.
\begin{equation}\label{aa}
 \left \{
\begin{array}{ll}
0< c<\nu <\frac{1}{8},\\
\\
0< \gamma < \beta <\frac{1}{2},\\
\\
0< \gamma < \nu -c,\\
\\
\vert z\vert \leq (\frac{\ln n}{(d-1)^{6}})^{c},\\
\\
t=\ln n+ \ln \frac{1}{\varepsilon},\\
\\
\theta(n)=\ln \ln n.
\end{array}\right.
\end{equation}
Also, for this we define 
\[
C_{j}:=\frac{1}{\sqrt{n^{d-1}}} \sum\limits_{i_{1},i_{3},\dots,i_{d}=1}^{n} a_{i_{1}j i_{3}\dots i_{d}},\qquad \forall j=1,\dots,n,
\]
where $a_{i_{1}j i_{3}\dots i_{d}}$ is the $(i_{1},j,i_{3},\dots,i_{d})$th entry of $A$, where $A$ is a random tensor with i.i.d entries sampled from $\mathcal{D}_{0} $. Also, we define
\begin{eqnarray}\label{rs2}
V_{k}:=\frac{1}{n^{\frac{k(d-1)}{2}}} e_{k}(C_{1},C_{2},\dots,C_{n}), \qquad \forall k=1,\dots,n
\end{eqnarray}
and
\[
D_{k}:=\frac{1}{n^{\frac{k(d-1)}{2}}} S_{k}(C_{1},C_{2},\dots,C_{n}),\qquad \forall k=1,\dots,n.
\]
Then in Section $4$, we show that we can approximate $\sum\limits_{k=0}^{t} a_{k} z^{k}$ by $\sum\limits_{k=0}^{t} V_{k}z^{k}$, in which $ t $ is mentioned in relation \eqref{aa}. In Section $6$, we observe that $ \sum\limits_{k=0}^{t} V_{k}z^{k} $ can be approximated by $ \sum\limits_{k=0}^{t} V^{\prime}_{k}z^{k} $. Also, we show that $ \sum\limits_{k=0}^{t} V^{\prime}_{k}z^{k} $ can be approximated by $ \sum\limits_{k=0}^{\infty} V^{\prime}_{k}z^{k} $. Then, we show that $ \sum\limits_{k=0}^{\infty} V^{\prime}_{k} z^{k}=e^{V_{1}z - \frac{\xi z^{2}}{2}}$, where $V_{1}$ is the normalization of the entries of tensor $A$ and this complete proof.  In Section $7$, we present details proof of theorem and we present a PTAS. 

 Now, if we relax the approximation requirement, we can compute $V_{1}$ and estimate $ per(X) $ by $(n!)^{d-1} e^{V_{1}z-\frac{\xi z^{2}}{2}}$, where $\xi$ is the quasi-variance of  $\mathcal{D}$.

 Notice that, for such choice of $ t $, the algorithm has time complexity $t^{d-1} (\binom{n}{t})^{d} (t!)^{d-1}= \mathcal{O}(n^{dt}) $, which is quasi-polynomial. 

\section{Estimation using the summation of hyperplanes of type $2$}
In this section, we show that $ \sum\limits_{k=t+1}^{n}a_{k}z^{k} $ and $\sum\limits_{k=0}^{t}(a_{k}-V_{k}) z^{k} $ are small with high probability. At first, we define the below notation for the sake of convenience.
 \begin{eqnarray}\label{m4}
 \Lambda (I^{1},I^{2},\dots ,I^{d}:J^{1},J^{2},\dots ,J^{d}):=\sum_{\substack{\lbrace i_{1}^{1},\dots, i_{k}^{1}\rbrace \in I^{1}\\
\lbrace i_{1}^{2},\dots, i_{k}^{2}\rbrace \in I^{2}\\
\vdots \\
\lbrace i_{1}^{d},\dots, i_{k}^{d}\rbrace \in I^{d}
}}
\sum_{\substack{\lbrace j_{1}^{1},\dots, j_{l}^{1}\rbrace \in J^{1}\\
\lbrace j_{1}^{2},\dots, j_{l}^{2}\rbrace \in J^{2}\\
\vdots \\
\lbrace j_{1}^{d},\dots, j_{l}^{d}\rbrace \in J^{d}
}}
E\Bigg[\prod\limits_{t=1}^{k} a_{i_{t}^{1}i_{t}^{2}\dots i_{t}^{d}}\prod\limits_{t=1}^{l} \overline{a_{j_{t}^{1}j_{t}^{2}\dots j_{t}^{d}}}\Bigg].
 \end{eqnarray}
 Since all the entries of $A$ are i.i.d and of zero mean, the expectation of $\prod\limits_{t=1}^{k} a_{i_{t}^{1}i_{t}^{2}\dots i_{t}^{d}}\prod\limits_{t=1}^{l} \overline{a_{j_{t}^{1}j_{t}^{2}\dots j_{t}^{d}}}
$ is non-zero if 
\begin{equation*}
\bigg\lbrace (i_{1}^{1},i_{1}^{2},\dots , i_{1}^{d}),\dots ,( i_{k}^{1},i_{k}^{2},\dots , i_{k}^{d}) \bigg\rbrace = \bigg\lbrace (j_{1}^{1},j_{1}^{2},\dots , j_{1}^{d}),\dots ,( j_{l}^{1},j_{l}^{2},\dots , j_{l}^{d}) \bigg\rbrace .
\end{equation*} 
Since $\Lambda (I^{1},I^{2},\dots ,I^{d}:J^{1},J^{2},\dots ,J^{d})=0$ for $k\neq l$, we can rewrite (\ref{m4}) as
\begin{eqnarray}\label{m5}
 \Lambda (I^{1},I^{2},\dots ,I^{d}:J^{1},J^{2},\dots ,J^{d})=\sum_{\substack{\lbrace i_{1}^{s},\dots, i_{k}^{s}\rbrace \in I^{s}\cap J^{s} }} \quad \prod\limits_{t=1}^{k} E \vert a_{i_{t}^{1}i_{t}^{2}\dots i_{t}^{d}}\vert^{2},
\end{eqnarray}
where $ s=1,2,\dots,d $.
\begin{lemma}\label{m3}
For every $k,l \in \lbrace 0,1,\dots,n\rbrace$, 
\begin{equation*}
E[a_{k}]=\delta_{k ,0},\qquad E[a_{k}\overline{a_{l}}]=\frac{\delta_{k,l}}{k! (n^{\underline{k}})^{d-2}}\leq \frac{1}{k!}. 
\end{equation*}
\end{lemma}
\begin{proof}
Since $a_{0}\equiv 1$, $E[a_{0}]=1$. For $k>0$,
\begin{eqnarray*}
\nonumber E[a_{k}]&=& \frac{1}{(n^{\underline{k}})^{d-1}}\sum\limits_{B\subseteq_{k} A}E[per(B)]\\
\nonumber &=& \frac{1}{(n^{\underline{k}})^{d-1}} \sum_{\substack{\lbrace i_{1},\dots, i_{k}\rbrace \in C_{n,k}\\
\lbrace j_{1}^{s},\dots, j_{k}^{s}\rbrace \in P_{n,k}
}}
E\Bigg[\sum\limits_{\sigma }\prod\limits_{t=1}^{k} a_{i_{t} j_{\sigma_{t}}^{1}\dots j_{\sigma_{t}}^{d-1}}\Bigg]\\
\nonumber &=& \frac{1}{(n^{\underline{k}})^{d-1}}\sum_{\substack{\lbrace i_{1},\dots, i_{k}\rbrace \in C_{n,k}\\
\lbrace j_{1}^{s},\dots, j_{k}^{s}\rbrace \in P_{n,k}\\
}}
E\Bigg[\prod\limits_{t=1}^{k} a_{i_{t} j_{\sigma_{t}}^{1}\dots j_{\sigma_{t}}^{d-1}}\Bigg]\\
\nonumber &=&\frac{1}{(n^{\underline{k}})^{d-1}}\sum_{\substack{\lbrace i_{1},\dots, i_{k}\rbrace \in C_{n,k}\\
\lbrace j_{1}^{s},\dots, j_{k}^{s}\rbrace \in P_{n,k}\\
}}\prod\limits_{t=1}^{k} E\Bigg[a_{i_{t} j_{\sigma_{t}}^{1}\dots j_{\sigma_{t}}^{d-1}}\Bigg]\\
&=&0,
\end{eqnarray*} 
where $ s=1,2,\dots,d-1 $ and the summation runs over all one-to-one functions $ \sigma $ from $ \lbrace 1, \dots , k \rbrace $ to $ \lbrace 1, \dots , k \rbrace $. The last equality is true since the entries of $A$ are i.i.d and have mean value equal to $0$. This proves that $E[a_{k}]=0$ for $k> 0$.
\\
Now, for the second part by using (\ref{m5}) we obtain
\begin{eqnarray*}
\nonumber E[a_{k}\overline{a_{l}}]&=& \frac{1}{(n^{\underline{k}})^{d-1}(n^{\underline{l}})^{d-1}} \sum\limits_{B\subseteq_{k} A}\sum\limits_{B^{\prime}\subseteq_{l} A} E [per(B) \overline{per(B^{\prime})}]\\
\nonumber &=& \frac{1}{(n^{\underline{k}})^{d-1}(n^{\underline{l}})^{d-1}}  \Lambda(C_{n,k},P_{n,k},P_{n,k},\dots , P_{n,k}: C_{n,l},P_{n,l},P_{n,l},\dots , P_{n,l})\\
\nonumber &=& \frac{\delta_{k,l}}{(n^{\underline{k}})^{2(d-1)}}  \sum_{\substack{\lbrace i_{1}^{1},\dots, i_{k}^{1}\rbrace \in C_{n,k} \\
\lbrace i_{1}^{l},\dots, i_{k}^{l}\rbrace \in P_{n,k}
}}
\prod\limits_{t=1}^{k} E \vert a_{i_{t}^{1}i_{t}^{2}\dots i_{t}^{d}}\vert^{2}\\
\nonumber &=&\frac{\delta_{k,l}}{(n^{\underline{k}})^{2(d-1)}}  \sum_{\substack{\lbrace i_{1}^{1},\dots, i_{k}^{1}\rbrace \in C_{n,k} \\
\lbrace i_{1}^{l},\dots, i_{k}^{l}\rbrace \in P_{n,k}
}}
\prod\limits_{t=1}^{k} Var(a_{i_{t}^{1}i_{t}^{2}\dots i_{t}^{d}})\\
\nonumber &=&\frac{\delta_{k,l}}{(n^{\underline{k}})^{2(d-1)}} \binom{n}{k}(\frac{n!}{(n-k)!})^{d-1}\\
\nonumber &=&\frac{1}{k! (n^{\underline{k}})^{(d-2)}}\\
\nonumber &\leq &\frac{1}{k!},
\end{eqnarray*}
where $ l=2,3,\dots,d $. Here we used the fact that, the variance of each entry is $1$.
\end{proof}
\begin{lemma}\label{m26}
For any $m \in \lbrace 0,1,\dots,n \rbrace$ and $k,l \in \lbrace 1,2,\dots,n\rbrace$,
\[
E[V_{m}]=\delta_{m,0}
\]
and
\begin{eqnarray*}
 \nonumber E[(V_{k}-a_{k})\overline{(V_{l}-a_{l})}]&=&\delta_{k,l}\Bigg[\bigg((\frac{1}{n^{(d-1)k}}-\frac{1}{(n^{\underline{k}})^{d-1}})^{2}\frac{(n^{\underline{k}})^{d}}{k!}\bigg)\\
\nonumber &+& \bigg( \frac{1}{k!n^{2k(d-1)}} \sum\limits_{i=1}^{d-1}(n^{\underline{k}})^{i}(n-n^{\underline{k}}) (n^{k})^{d-1-i}\bigg)\Bigg]\\
\nonumber &\leq &\frac{k(k-1)(d-1)^{2}}{4n^{2}(k-2)!}+\frac{d-1}{2n (k-2)!}.
\end{eqnarray*}
\end{lemma}
\begin{proof}
Since $V_{0}\equiv1$ by definition, $E[V_{0}]=1$. For $m> 0$, 
\begin{eqnarray*}
\nonumber E[V_{m}]&=&\frac{1}{n^{\frac{m(d-1)}{2}}} \sum\limits_{\lbrace j_{1},\dots ,j_{m}\rbrace \in C_{n,m}} E[C_{j_{1}} C_{j_{2}} \dots C_{j_{m}}]\\
\nonumber &=& \frac{1}{n^{\frac{m(d-1)}{2}}} \sum\limits_{\lbrace j_{1},\dots ,j_{m}\rbrace \in C_{n,m}} E \bigg [ (\frac{1}{n^{\frac{d-1}{2}}} \sum\limits_{i_{1}^{1},i_{3}^{1},\dots ,i_{d}^{1}=1}^{n} a_{i_{1}^{1} j_{1} i_{3}^{1} \dots i_{d}^{1}})\cdot \dots \\
\nonumber &\cdot & (\frac{1}{n^{\frac{d-1}{2}}} \sum\limits_{i_{1}^{m},i_{3}^{m},\dots ,i_{d}^{m}=1}^{n} a_{i_{1}^{m} j_{m} i_{3}^{m} \dots i_{d}^{m}})\bigg]\\
\nonumber &=& \frac{1}{n^{m(d-1)}} 
\sum_{\substack{ \lbrace j_{1},\dots ,j_{m}\rbrace \in C_{n,m}\\
\lbrace i_{1}^{s},i_{3}^{s},\dots ,i_{d}^{s}\rbrace \in [n] 
}} E\bigg [\prod_{t=1}^{m}a_{i_{1}^{t} j_{t} i_{3}^{t} \dots i_{d}^{t}}\bigg]\\
\nonumber &=& 0,
\end{eqnarray*}
where $ s=1,2,\dots,m $. For the second part by definitions (\ref{rs1}) and (\ref{rs2}), we can write
\begin{eqnarray}\label{mnn}
\nonumber V_{k}-a_{k}&=& \Bigg[ (\frac{1}{n^{k(d-1)}}-\frac{1}{(n^{\underline{k}})^{d-1}})
(\sum_{\substack{ \lbrace i_{1}^{1},\dots ,i_{k}^{1}\rbrace \in P_{n,k}\\
\lbrace i_{1}^{2},\dots ,i_{k}^{2}\rbrace \in C_{n,k} \\
\lbrace i_{1}^{l},\dots ,i_{k}^{l}\rbrace \in P_{n,k}\\
}} \prod\limits_{t=1}^{k} a_{i_{t}^{1}i_{t}^{2}\dots i_{t}^{d}}) \Bigg]
\\
\nonumber &+& \frac{1}{n^{k(d-1)}} \Bigg[ \sum_{\substack{ \lbrace i_{1}^{1},\dots ,i_{k}^{1}\rbrace \in P_{n,k}\\
\lbrace i_{1}^{2},\dots ,i_{k}^{2}\rbrace \in C_{n,k} \\
\lbrace i_{1}^{s},\dots ,i_{k}^{s}\rbrace \in P_{n,k}\\
\lbrace i_{1}^{d},\dots ,i_{k}^{d}\rbrace \in [n]^{k}-P_{n,k}\\
}} \prod\limits_{t=1}^{k} a_{i_{t}^{1}i_{t}^{2}\dots i_{t}^{d}}
+ 
\sum_{\substack{ \lbrace i_{1}^{1},\dots ,i_{k}^{1}\rbrace \in P_{n,k}\\
\lbrace i_{1}^{2},\dots ,i_{k}^{2}\rbrace \in C_{n,k} \\
\lbrace i_{1}^{r},\dots ,i_{k}^{r}\rbrace \in P_{n,k} \\
\lbrace i_{1}^{d-1},\dots ,i_{k}^{d-1}\rbrace \in [n]^{k}-P_{n,k}\\
\lbrace i_{1}^{d},\dots ,i_{k}^{d}\rbrace \in [n]^{k}\\
}} \prod\limits_{t=1}^{k} a_{i_{t}^{1}i_{t}^{2}\dots i_{t}^{d}}\\
 &+& \dots +
\sum_{\substack{ \lbrace i_{1}^{1},\dots ,i_{k}^{1}\rbrace \in [n]^{k}-P_{n,k}\\
\lbrace i_{1}^{2},\dots ,i_{k}^{2}\rbrace \in C_{n,k} \\
\lbrace i_{1}^{l},\dots ,i_{k}^{l}\rbrace \in [n]^{k} \\
}} \prod\limits_{t=1}^{k} a_{i_{t}^{1}i_{t}^{2}\dots i_{t}^{d}}\Bigg].
\end{eqnarray}
where $ l=3,4,\dots,d $, $ s=3,4,\dots,d-1 $ and $ r=3,4,d-2 $ . For computing $ E[(V_{k}-a_{k})\overline{(V_{l}-a_{l})}] $, at first we write $ V_{k}-a_{k} $ and $ V_{l}-a_{l} $ by using  (\ref{mnn}) and multiply these together. By (\ref{m5}), $ E[(V_{k}-a_{k})\overline{(V_{l}-a_{l})}]$ is non-zero only if $k=l$, and it is equal to $0$ for $k\neq l$. For $l=k$, we obtain
\begin{eqnarray}\label{m6}
\nonumber  E[(V_{k}-a_{k})\overline{(V_{k} - a_{k})}] &=&\Bigg[(\frac{1}{n^{k(d-1)}} - \frac{1}{(n^{\underline{k})^{d-1}}} )^{2}\\
\nonumber &\cdot &\Lambda(P_{n,k},C_{n,k},P_{n,k},\dots ,P_{n,k}:P_{n,k},C_{n,k},P_{n,k},\dots ,P_{n,k}) \Bigg]\\
\nonumber &+& \Bigg[\frac{1}{n^{2k(d-1)}} \\
\nonumber &\cdot & \scalebox{.95}{$\Lambda(P_{n,k},C_{n,k},P_{n,k},\dots ,P_{n,k},[n]^{k} - P_{n,k}: P_{n,k},C_{n,k},P_{n,k},\dots ,P_{n,k},[n]^{k} - P_{n,k}) $}\Bigg]\\
\nonumber &+& \Bigg[\frac{1}{n^{2k(d-1)}}\\
\nonumber &\cdot &  \scalebox{.9}{$\Lambda(P_{n,k},C_{n,k},P_{n,k},\dots ,P_{n,k},[n]^{k} - P_{n,k},[n]^{k}: P_{n,k},C_{n,k},P_{n,k},\dots ,P_{n,k},[n]^{k} - P_{n,k},[n]^{k})$}\Bigg]\\
\nonumber &+& \dots + \Bigg[ \frac{1}{n^{2k(d-1)}} \\
\nonumber &\cdot &\Lambda([n]^{k} - P_{n,k},C_{n,k},[n]^{k},[n]^{k},\dots ,[n]^{k}: [n]^{k} - P_{n,k},C_{n,k},[n]^{k},[n]^{k},\dots ,[n]^{k})\Bigg].\\ 
\end{eqnarray}
Thus, by relation (\ref{m6}) we can write
\begin{eqnarray}\label{m7}
\nonumber E[(V_{k}-a_{k})\overline{(V_{k} - a_{k})}]&=& \Bigg[\bigg (\frac{1}{n^{k(d-1)}} - \frac{1}{(n^{\underline{k})^{d-1}}} \bigg)^{2}
\sum_{\substack{ \lbrace i_{1}^{1},\dots ,i_{k}^{1}\rbrace \in P_{n,k}\\
\lbrace i_{1}^{2},\dots ,i_{k}^{2}\rbrace \in C_{n,k} \\
\lbrace i_{1}^{l},\dots ,i_{k}^{l}\rbrace \in P_{n,k}\\
}} \prod\limits_{t=1}^{k} a_{i_{t}^{1}i_{t}^{2}\dots i_{t}^{d}}\Bigg]\\
\nonumber &+&  \frac{1}{n^{2k(d-1)}}\Bigg [\sum_{\substack{ \lbrace i_{1}^{1},\dots ,i_{k}^{1}\rbrace \in P_{n,k}\\
\lbrace i_{1}^{2},\dots ,i_{k}^{2}\rbrace \in C_{n,k} \\
\lbrace i_{1}^{s},\dots ,i_{k}^{s}\rbrace \in P_{n,k}\\
\lbrace i_{1}^{d},\dots ,i_{k}^{d}\rbrace \in [n]^{k}-P_{n,k}\\
}} \prod\limits_{t=1}^{k} a_{i_{t}^{1}i_{t}^{2}\dots i_{t}^{d}} + \dots \\
\nonumber &+& \sum_{\substack{ \lbrace i_{1}^{1},\dots ,i_{k}^{1}\rbrace \in [n]^{k}-P_{n,k}\\
\lbrace i_{1}^{2},\dots ,i_{k}^{2}\rbrace \in C_{n,k} \\
\lbrace i_{1}^{l},\dots ,i_{k}^{l}\rbrace \in [n]^{k} \\
}} \prod\limits_{t=1}^{k} a_{i_{t}^{1}i_{t}^{2}\dots i_{t}^{d}}\Bigg]\\
\nonumber &=& (\frac{1}{n^{k(d-1)}} - \frac{1}{(n^{\underline{k})^{d-1}}} )^{2} \frac{(n^{\underline{k}})^{d}}{k!}+ \frac{1}{n^{2k(d-1)}k!}. (n^{\underline{k}})^{d-1}\cdot (n^{k}-n^{\underline{k}})\\
\nonumber &+& \dots +\frac{1}{n^{2k(d-1)}k!} n^{\underline{k}} (n^{k}-n^{\underline{k}}) n^{k(d-2)}\\
 &=&(\frac{1}{n^{k(d-1)}} - \frac{1}{(n^{\underline{k})^{d-1}}} )^{2} \frac{(n^{\underline{k}})^{d}}{k!} + \sum\limits_{i=1}^{d-1} \dfrac{(n^{\underline{k}})^{i}(n^{k}-n^{\underline{k}}) n^{k(d-1-i)}}{n^{2k(d-1)}k!},
\end{eqnarray}
where $ l=3,4,\dots,d $, $ s=3,4,\dots,d-1 $. In what follows, we bound each term of equation (\ref{m7}).
\itemize
\item
First,
\begin{eqnarray*}
\nonumber \bigg(\frac{1}{n^{k(d-1)}} - \frac{1}{(n^{\underline{k})^{d-1}}}\bigg )^{2} \frac{(n^{\underline{k}})^{d}}{k!}&=& \frac{(n^{k(d-1)}-(n^{\underline{k}})^{d-1})^{2}}{n^{2k(d-1)}.(n^{\underline{k}})^{2(d-1})}.\frac{(n^{\underline{k}})^{d}}{k!}\\
\nonumber &=& \bigg(\frac{n^{k(d-1)}-(n^{\underline{k}})^{d-1}}{n^{k(d-1)}}\bigg)^{2}.\frac{1}{k!(n^{\underline{k}})^{d-2}}\\
\nonumber &\leq &\bigg (1-\bigg[(1-\frac{1}{n})(1-\frac{2}{n})\dots (1-\frac{k-1}{n})\bigg]^{d-1}\bigg)^{2}\\
\nonumber &\cdot &\frac{1}{k!(n^{\underline{k}})^{d-2}}\\
\nonumber &=& \bigg (1- (1-\frac{1}{n})(1-\frac{2}{n})\dots (1-\frac{k-1}{n})\bigg)^{2}\\
\nonumber &\cdot & \bigg(\sum\limits_{s=0}^{d-2} (1-\frac{1}{n})^{d-s-2}\dots (1-\frac{k-1}{n})^{d-s-2} \bigg)^{2}\times \frac{1}{k!(n^{\underline{k}})^{d-2}}.
\end{eqnarray*}
Since 
\[
\bigg (1- (1-\frac{1}{n})(1-\frac{2}{n})\dots (1-\frac{k-1}{n})\bigg)^{2}\leq  \bigg(\sum\limits_{t=0}^{k-1}\frac{t}{n}\bigg)^{2}\leq \bigg(\frac{k(k-1)}{2n}\bigg)^{2}, 
\]
we can write
\begin{eqnarray*}
& &\bigg (1- (1-\frac{1}{n})(1-\frac{2}{n})\dots (1-\frac{k-1}{n})\bigg)^{2}\\
&\cdot &  \bigg(\sum\limits_{s=0}^{d-2} (1-\frac{1}{n})^{d-s-2}\dots (1-\frac{k-1}{n})^{d-s-2} \bigg)^{2}\times \frac{1}{k!(n^{\underline{k}})^{d-2}}\\
 &\leq & \bigg(\frac{k(k-1)}{2n}\bigg)^{2}\\
 &\cdot & \bigg(\sum\limits_{s=0}^{d-2} (1-\frac{1}{n})^{d-s-2}\dots (1-\frac{k-1}{n})^{d-s-2} \bigg)^{2}\times \frac{1}{k!(n^{\underline{k}})^{d-2}}\\
 &\leq & \frac{(k(k-1))^{2}}{4n^{2}}(d-1)^{2}.\frac{1}{k!}.
\end{eqnarray*}
Thus, we have
\begin{eqnarray}\label{m8}
\bigg(\frac{1}{n^{k(d-1)}} - \frac{1}{(n^{\underline{k})^{d-1}}}\bigg )^{2} \frac{(n^{\underline{k}})^{d}}{k!} \leq \frac{(k(k-1))^{2}}{4n^{2}}(d-1)^{2}.\frac{1}{k!}.
\end{eqnarray}
\item
Second,
\begin{eqnarray}\label{m99}
\nonumber \sum\limits_{i=1}^{d-1}\frac{1}{k!}.\frac{(n^{k}-n^{\underline{k}})}{n^{k}}. \frac{(n^{\underline{k}})^{i}n^{k(d-1-i)}}{n^{2k(d-1)-k}}&=& \sum\limits_{i=1}^{d-1} \frac{1}{k!}. \frac{(n^{k}-n^{\underline{k}})}{n^{k}}. \frac{(n^{\underline{k}})^{i}}{n^{kd+ki-2k}}\\
\nonumber &\leq & \sum\limits_{i=1}^{d-1} \frac{1}{k!}.\frac{(n^{k}-n^{\underline{k}})}{n^{k}}. \frac{(n^{\underline{k}})^{i}}{n^{ki}n^{k(d-2)}}\\
\nonumber &\leq & (d-1)\frac{1}{k!}. \frac{k(k-1)}{2n}.\frac{1}{n^{k(d-2)}}\\
\nonumber &\leq & (d-1)\frac{1}{k!}. \frac{k(k-1)}{2n}.\\
\end{eqnarray}
The last inequality follows from the facts $d-2>0$, $(n^{\underline{k}})^{i}\leq n^{ki}$ and $n^{k(d-2)}>1$. Thus, by (\ref{m7}), (\ref{m8}) and (\ref{m99}) we obtain
\begin{eqnarray}
\nonumber E[(V_{k}-a_{k})\overline{(V_{l} - a_{l})}] &\leq & \frac{(k(k-1))^{2}}{4n^{2}}(d-1)^{2}.\frac{1}{k!}+\frac{1}{k!}\frac{k(k-1)}{2n}(d-1)\\
\nonumber &\leq & \frac{k(k-1)(d-1)^{2}}{4n^{2}(k-2)!}+\frac{d-1}{2n (k-2)!}.
\end{eqnarray}
\end{proof}
In the following lemma, we prove that the remaining terms of the series are small with high probability.
\begin{lemma}\label{12} 
Taking relations (\ref{aa}) into account, we find that
\begin{equation*}
P \Bigg( \bigg \vert \sum\limits_{k=t+1}^{n} a_{k} z^{k} \bigg \vert \leq n^{-\gamma} \varepsilon \Bigg) \geq 1-o(1).
\end{equation*}
\end{lemma}
\begin{proof}
First, it follows from Lemma \ref{m3} that
\[
E \bigg [ \sum\limits_{k=t+1}^{n} a_{k} z^{k} \bigg]=\sum\limits_{k=t+1}^{n} E[a_{k}] z^{k}=0
\]
and 
\begin{eqnarray}
\nonumber Var \bigg (\sum\limits_{k=t+1}^{n} a_{k} z^{k} \bigg )&=& E\Bigg [ \bigg \vert \sum\limits_{k=t+1}^{n}a_{k} z^{k} \bigg \vert^{2} \Bigg]\\
\nonumber &=& \sum\limits_{k,l=t+1}^{n} E[a_{k} \overline{a_{l}} ] z^{k} \overline{z}^{l}\\
\nonumber &=&\sum\limits_{k=t+1}^{n}  \frac{1}{k!}. \frac{1}{(n^{\underline{k})^{d-2}}} . \vert z \vert ^{2k}\\
\nonumber & \leq & \sum\limits_{k=t+1}^{n}  \frac{1}{k!} \vert z \vert ^{2k}.
\end{eqnarray}
Now, by Chebyshev's inequality we obtain
\begin{eqnarray}
\nonumber P \Bigg( \bigg \vert \sum\limits_{k=t+1}^{n} a_{k} z^{k} \bigg \vert \geq n^{-\gamma} \varepsilon \Bigg) & \leq & \frac{n^{2 \gamma}}{\varepsilon^{2}} Var \bigg( \sum\limits_{k=t+1}^{n} a_{k} z^{k} \bigg )\\
\nonumber & \leq & \frac{n^{2 \gamma}}{\varepsilon^{2}}  \sum\limits_{k=t+1}^{n} \frac{\vert z \vert^{2k}}{k!}. 
\end{eqnarray}
For each $k \geq t $, 
\[
\frac{\vert z \vert^{2(k+1)}}{(k+1)!}\div \frac{\vert z \vert^{2k}}{k!}< \frac{1}{2(d-1)^{6}}.
\]
Thus,
\begin{eqnarray}
\nonumber P \Bigg( \bigg \vert \sum\limits_{k=t+1}^{n} a_{k} z^{k} \bigg \vert \geq n^{-\gamma} \varepsilon\bigg) &\leq & \frac{n^{2\gamma}}{\varepsilon^{2}} \bigg (\frac{\vert z \vert ^{2(t+1)}}{(t+1)!}+ 
\frac{\vert z \vert ^{2(t+2)}}{(t+2)!}+ \dots + \frac{\vert z \vert ^{2n}}{n!}\bigg)\\
\nonumber & \leq & \frac{n^{2\gamma}}{\varepsilon^{2}}\bigg ( \frac{1}{2(d-1)^{6}}\frac{\vert z \vert^{2t}}{t!}+ \frac{1}{4(d-1)^{12}}\frac{\vert z \vert^{2t}}{t!}+\dots + \frac{1}{2^{n-t}(d-1)^{6(n-t)}}\frac{\vert z \vert^{2t}}{t!}\bigg)\\
\nonumber & \leq & \frac{n^{2\gamma}}{\varepsilon^{2}}\frac{\vert z \vert^{2t}}{t!} 
 \sum\limits_{k=1}^{\infty} (\frac{1}{2(d-1)^{6}})^{2k}\\
\nonumber & \leq &\frac{n^{2\gamma}}{\varepsilon^{2}}\frac{\vert z \vert^{2t}}{t!} \bigg( \frac{1}{2(d-1)^{6}-1}\bigg ).
\end{eqnarray}
Here we used the fact that, the power series converges to $ \frac{1}{2(d-1)^{6}-1}$ for large $d$. Also, since $t=\ln \frac{n}{\varepsilon}$ and $ \vert z \vert\leq \bigg(\frac{\ln n}{(d-1)^{6}}\bigg)^{\frac{1}{8}}$, this probability is $o(1)$ for large $n$ and $d$.
\end{proof}
\begin{lemma}\label{16}
Taking relations (\ref{aa}) into account, with probability $ 1-o(1) $,
\[
\bigg \vert \sum_{k=0}^{t} a_{k}z^{k} - \sum_{k=0}^{t} V_{k}z^{k} \bigg\vert \leq n^{-\beta}.
\]
\end{lemma}
\begin{proof}
By Lemma \ref{m3} and Lemma \ref{m26},
\[
E[ V_{k}]=E[a_{k}]=\delta_{k,0},\qquad V_{0}=a_{0}\equiv 1.
\]
Thus, Lemma \ref{m26} allows us to write
\begin{eqnarray}
\nonumber Var \bigg [\sum_{k=0}^{t} a_{k}z^{k} - \sum_{k=0}^{t} V_{k}z^{k} \bigg ] &=& E\bigg [\bigg ( \sum_{k=0}^{t} (a_{k}-V_{k}) z^{k}\bigg) \bigg( \sum_{k=0}^{t} \overline{(a_{k}-V_{k})} \overline{z}^{k} \bigg ]\\
\nonumber &=& \sum_{k=0}^{t} E \bigg [ (V_{k}-a_{k}) \overline{(V_{k}-a_{k})} \bigg ] \vert z\vert^{2k}\\
\nonumber &\leq & \sum_{k=0}^{t} \vert z\vert^{2k} \bigg(\frac{k(k-1)(d-1)^{2}}{4n^{2}(k-2)!}+\frac{d-1}{2n (k-2)!}\bigg)\\
\nonumber &\leq & \sum_{m=0}^{t-2} \bigg (\dfrac{\vert z\vert^{2(m+2)}}{m! 4n^{2}}(m+2)(m+1)(d-1)^{2} + \frac{\vert z\vert^{2(m+2)}}{m! 2n}(d-1) \bigg)\\
\nonumber &\leq & \frac{\vert z \vert^{4}}{2n} \bigg(\sum_{m=0}^{\infty} \bigg( \frac{\vert z \vert^{2m}}{m!} ) (d-1)^{2}+ \frac{\vert z \vert^{2m}}{m!}(d-1)^{2} \bigg)\\
\nonumber &=& \frac{\vert z \vert^{4}}{n} (d-1)^{2} \bigg(\sum_{m=0}^{\infty}  \frac{\vert z \vert^{2m}}{m!} \bigg) \\
\nonumber &=& \frac{\vert z \vert^{4} e^{(\vert z \vert)^{2}} }{n} (d-1)^{2}. 
\end{eqnarray}
Here, we used the Taylor expansion of $e^{(\vert z \vert)^{2}}$ and the facts that $d> 2$ and $n> 1$. Now, by Chebyshev's inequality we obtain
\[
P \bigg( \bigg \vert \sum_{k=0}^{t} a_{k}z^{k} - \sum_{k=0}^{t} V_{k}z^{k} \bigg\vert \geq n^{-\beta} \bigg)
\leq \frac{\vert z \vert^{4} e^{(\vert z \vert)^{2}}}{n^{1-2\beta}} (d-1)^{2}.
\]
Since $ \vert z \vert \leq (\frac{ln n}{(d-1)^{6}})^{\frac{1}{8}}$, $ \beta < \frac{1}{2}$ and $d> 2$, this probability is $o(1)$ for large $n$ and $d$.
\end{proof}

\section{Approximation of $D_{2}$} 
In this section, we show that $D_{2}$ can be approximated by $\xi$, where $\xi$ is the quasi-variance of $\mathcal{D}$. Also, we present a bound for $D_{k}$, where $ k \geq 3$.
\begin{lemma}\label{m22}
For each $0< \phi < \frac{2d-3}{2}$ with probability $ 1-o(1) $,
\[
\vert D_{2}-\xi \vert \leq n^{-\phi}.
\]
\end{lemma}
\begin{proof}
First, we define 
\begin{equation*}
 \left \{
\begin{array}{ll}
X_{i_{1}\dots i_{d}}:= a_{i_{1}i_{2}\dots i_{d}} 1_{\vert a_{i_{1}\dots i_{d}}\vert \leq n^{d-1}},\\
\\
\mu_{k}:= E[X_{i_{1}\dots i_{d}}^{k}],\\
\\
\mu_{k}^{*}:= E[\vert X_{i_{1}\dots i_{d}}\vert ^{k}],\\
\\
\mu^{\dagger}:= E[\vert X_{i_{1}\dots i_{d}} \vert^{2}X_{i_{1}\dots i_{d}}].
\end{array} \right.
\end{equation*}
These values are well-defined, because all elements of $A$ are i.i.d and $X_{i_{1}\dots i_{d}}$ is bounded for every $i_{1},\dots,\dots i_{d}$. We assume that $n^{d-1}\geq \rho $ in what follows. If $X$ is a random variable with distribution $\mathcal{D}$ and 
\[
\sigma_{1}^{2}=E_{X \sim \mathcal{D}}[Re(X)^{2}], \qquad \sigma_{2}^{2}=E_{X \sim \mathcal{D}}[Im(X)^{2}],
\]
then
\[
\rho \geq (\sigma_{1}^{2} + \sigma_{2}^{2})^{\frac{3}{2}}=1.
\]
Thus, $\rho \geq 1$. Now, by Chebyshev's inequality we obtain
\begin{eqnarray}
\nonumber P(\vert a_{i_{1}\dots i_{d}} \vert > n^{d-1}) &\leq &\frac{1}{n^{2(d-1)}}\\
\nonumber &\leq & \frac{\rho}{n^{3(d-1)}}.
\end{eqnarray}
$A$ satisfies
\[
P(\exists i_{1},\dots,i_{d} \in [n]: \vert a_{i_{1},\dots,i_{d}} \vert >n^{d-1})\leq \frac{\rho}{n^{d-1}}.
\]
Thus,
\begin{equation}\label{m9}
P(\vert D_{2}-\xi \vert >\varepsilon) \leq P\Bigg(\Bigg \vert \frac{\sum\limits_{i_{2}=1}^{n}\bigg(\sum\limits_{i_{1},i_{3},i_{4},\dots ,i_{d}=1}^{n} X_{i_{1}\dots i_{d}}\bigg)^{2}}{n^{2(d-1)}}-\xi \Bigg \vert > \varepsilon \Bigg) +\frac{\rho}{n^{d-1}}.
\end{equation}
In the sequel, we prove some useful relations. 
 \begin{eqnarray}\label{m10}
\nonumber \vert \mu_{1} \vert = \vert E[X_{i_{1}\dots i_{d}}]\vert & =& \Bigg \vert E\bigg [a_{i_{1}\dots i_{d}}1_{\vert a_{i_{1}\dots i_{d}\vert \leq n^{d-1}}}\bigg] \Bigg \vert\\
\nonumber &=&\Bigg \vert - E \bigg[a_{i_{1}\dots i_{d}} 1_{\vert a_{i_{1}\dots i_{d}} \vert > n^{d-1}}\bigg] \Bigg \vert\\
\nonumber &\leq & E \Bigg [\bigg \vert a_{i_{1}\dots i_{d}} \bigg \vert 1_{\vert a_{i_{1}\dots i_{d}} \vert > n^{d-1}}\Bigg]\\
\nonumber &\leq & E \Bigg [ \vert a_{i_{1}\dots i_{d}} \vert \bigg(\frac{\vert a_{i_{1}\dots i_{d}}\vert }{n^{d-1}}\bigg)^{2}\Bigg ]\\
&\leq &\frac{\rho}{n^{2(d-1)}}.
\end{eqnarray} 
 To find the upper bound of $\mu_{2}$, first notice that 
\begin{eqnarray}\label{m11}
\nonumber \vert \xi -\mu_{2}\vert &=&\bigg \vert \xi - E[ X_{i_{1}\dots i_{d}}^{2}] \bigg \vert \\
\nonumber &=& \Bigg \vert \xi - E \bigg [ a_{i_{1}\dots i_{d}}^{2}  1_{\vert a_{i_{1}\dots i_{d}}\vert \leq n^{d-1}} \bigg ] \Bigg \vert \\
\nonumber & = & \Bigg \vert E [ a_{i_{1}\dots i_{d}}^{2}] - E [a_{i_{1}\dots i_{d}}^{2} 1_{\vert a_{i_{1} \dots i_{d}} \vert \leq n^{d-1}} ] \Bigg \vert \\
\nonumber & =& \Bigg \vert E [a_{i_{1}\dots i_{d}}^{2} 1_{\vert a_{i_{1} \dots i_{d}} \vert > n^{d-1}} ] \Bigg \vert \\
\nonumber & \leq & E [ \vert a_{i_{1}\dots i_{d}}\vert^{2} 1_{\vert a_{i_{1} \dots i_{d}} \vert > n^{d-1}} ]\\
 & =& \frac{\rho}{n^{d-1}}.
\end{eqnarray}
Thus, by (\ref{m11}) we obtain
\begin{eqnarray}
\nonumber \vert \mu_{2} \vert - \vert \xi \vert \leq  \vert \xi - \mu_{2} \vert \leq \frac{\rho}{n^{d-1}}.
\end{eqnarray}
Hence, 
\begin{eqnarray}\label{s1}
 \vert \mu_{2} \vert  \leq \vert \xi \vert + \frac{\rho}{n^{d-1}} \leq 2.
\end{eqnarray}
Here, we used the fact $\frac{\rho}{n^{d-1}}\leq 1$ and relation (\ref{m1}). Now, we find upper bounds for $\mu_{2}^{*},\mu_{4}^{*}$ and $\mu^{\dagger}$.
\begin{eqnarray}\label{s2}
 \mu_{2}^{*}=E \bigg [\vert a_{i_{1}\dots i_{d}}\vert^{2} 1_{\vert a_{i_{1} \dots i_{d}}\vert \leq n^{d-1}} \bigg ]   \leq E [ \vert a_{i_{1} \dots i_{d}} \vert^{2} ] =1.
\end{eqnarray}
Also,
\begin{eqnarray}\label{s3}
 \mu_{4}^{*}=E \bigg [ \vert a_{i_{1}\dots i_{d}} \vert^{4} 1_{\vert a_{i_{1} \dots i_{d}}\vert \leq n^{d-1}} \bigg ] \leq n^{d-1}E \bigg [ \vert a_{i_{1}\dots i_{d}} \vert^{3}  \bigg] \leq \rho n^{d-1}.
\end{eqnarray}
Here we used the fact that, the entries of $A$ are i.i.d. Now,
\begin{eqnarray}\label{s4}
 \vert \mu^{\dagger}\vert =\Bigg \vert E \bigg [ \vert X_{i_{1}\dots i_{d}} \vert^{2}  X_{i_{1}\dots i_{d}} \bigg ] \Bigg \vert =\Bigg \vert E \bigg [ \vert a_{i_{1}\dots i_{d}} \vert^{2}  a_{i_{1}\dots i_{d}} 1_{\vert a_{i_{1} \dots i_{d}}\vert \leq n^{d-1}} \bigg] \Bigg \vert  \leq  E  [ \vert a_{i_{1}\dots i_{d}} \vert^{3}] =\rho.
\end{eqnarray}
We define $S_{i_{2}}:=\sum\limits_{i_{1},i_{3},\dots,i_{d}=1}^{n} X_{i_{1}\dots i_{d}} $ and $S:=\sum\limits_{i_{2}=1}^{n}(S_{i_{2}}^{2}-n^{d-1}\xi)$. Then,
\begin{eqnarray}\label{m13}
\nonumber E [\vert S \vert^{2}] &=& Var[S]+\vert E[S] \vert^{2}\\
\nonumber & = & Var \bigg [\sum\limits_{i_{2}=1}^{n}(S_{i_{2}}^{2}-n^{d-1}\xi)\bigg ]+ \bigg \vert E \bigg [\sum\limits_{i_{2}=1}^{n}(S_{i_{2}}^{2}-n^{d-1}\xi)\bigg ] \bigg \vert^{2}\\
\nonumber & \leq & n Var [S_{1}^{2}-n^{d-1}\xi ] + n^{2} \bigg \vert E[S_{1}^{2}-n^{d-1}\xi] \bigg\vert^{2}\\
\nonumber & = & n Var [S_{1}^{2}] + n^{2}\bigg \vert E[S_{1}^{2}-n^{d-1}\xi] \bigg\vert^{2}\\
& \leq & n E[S_{1}^{2}\overline{S_{1}}^{2}] + n^{2}\bigg \vert E[S_{1}^{2}-n^{d-1}\xi] \bigg \vert^{2}.
\end{eqnarray}
Here we used the fact $ Var(aX+b)= a^{2}Var(X)$, where $a$ and $b$ are constants and $X$ is a random variable. Now, we find bounds for $ E[S_{1}^{2}\overline{S_{1}}^{2}]$ and $\bigg \vert E[S_{1}^{2}-n^{d-1}\xi] \bigg \vert$.
\\
 For the first part, we note that 
 \begin{eqnarray}\label{m14}
\nonumber \bigg \vert E[S_{1}^{2}-n^{d-1}\xi] \bigg \vert &=& \bigg \vert E \bigg [( \sum\limits_{i_{1},i_{3},\dots ,i_{d}=1}^{n}X_{i_{1} 1 \dots i_{d}})^{2} -n^{d-1}\xi \bigg] \bigg \vert \\
\nonumber &=&  \bigg \vert n^{d-1}\mu_{2} + \bigg (n^{2d-3}(n-1)+n^{2d-4}(n-1)+ \dots +n^{d}(n-1) \bigg) \mu_{1}^{2} - n^{d-1} \xi \bigg \vert\\
\nonumber &\leq & n^{d-1} \vert \mu_{2}-\xi \vert +  (n^{2d-2}+n^{2d-3}+ \dots +n^{2d-(d-1)} ) \vert \mu_{1}\vert^{2} \\
\nonumber &\leq & n^{d-1}\frac{\rho}{n^{d-1}}+(d-2) n^{2d-2}\frac{\rho^{2}}{n^{4d-4}}\\
\nonumber &=& \rho +(d-2).
 \end{eqnarray} 
Here, we used (\ref{m11}) and (\ref{m10}). To compute $ E[S_{1}^{2}\overline{S_{1}}^{2}]$ by according to positions of indices in $X_{i_{1}\dots i_{d}},  X_{j_{1}\dots j_{d}}$, $X_{k_{1}\dots k_{d}}$ and $X_{l_{1}\dots l_{d}} $ in multiplication
$ S_{1}^{2}\overline{S_{1}}^{2} $, we consider five cases in below form. For convenience, we denote $ i_{1}1\dots i_{d} $ with $I$ and denote $ j_{1}1\dots j_{d} $ with $J$.
\itemize
\item
We first prove the following relation and consider $d\geq 3$.
\begin{eqnarray}\label{m21}
\nonumber \sum \limits_{i_{1},i_{3},\dots , i_{d}=1}^{n} E \bigg [X_{I}^{2} \overline{X_{I}}^{2}\bigg ] &=& n^{d-1}\mu_{4}^{*}\\
\nonumber &\leq & n^{2(d-1)} \rho \\
&\leq & n^{2(d-1)}(d-2)^{2} \rho^{2}.
\end{eqnarray}
Here, we used relation (\ref{s3}). 
\item
By according to multiplication of $ S_{1}^{2}\overline{S_{1}}^{2} $, we understand the indices of  $X_{i_{1}1\dots i_{d}}$ and $ X_{j_{1}1\dots j_{d}}$ that are placed together in the below form, they have property $ (i_{1},i_{3},\dots,i_{d}) \neq  (j_{1},i_{3},\dots,j_{d})$. Thus, there is at least $ l \in \lbrace 1,3,\dots,d \rbrace$ that
\[
 i_{l}\neq j_{l}, \quad 1 \leq j_{k} \leq n \quad \forall k=3,\dots,d, \quad k\neq l
 \]

Now, we obtain
\begin{eqnarray}\label{m15}
\nonumber \sum\limits_{\exists  l, i_{l}\neq j_{l}} 
 &E&\bigg[ X_{I} X_{J} \overline{X_{J}}. \overline{X_{J}} +  X_{J} X_{I}  \overline{X_{J}} . \overline{X_{J}} + X_{J} X_{J} \overline{X_{I}}. \overline{X_{J}} + X_{J} X_{J} \overline{X_{J}} .\overline{X_{I}}\bigg] \\
\nonumber &=& \sum_{\substack { i_{1}\neq j_{1}}} E \bigg [ X_{I} X_{J} \overline{X_{J}}. \overline{X_{J}} 
+  X_{J} X_{I}  \overline{X_{J}} . \overline{X_{J}} + X_{J} X_{J} \overline{X_{I}}. \overline{X_{J}} + X_{J} X_{J} \overline{X_{J}} .\overline{X_{I}}\bigg] \\
\nonumber &+& \dots  +\sum_{\substack { i_{d}\neq j_{d}}}
E \bigg [X_{I} X_{J} \overline{X_{J}}.\overline{X_{J}} +  X_{J} X_{I}  \overline{X_{J}} . \overline{X_{J}}  + X_{J} X_{J} \overline{X_{I}}.\overline{X_{J}} + X_{J} X_{J} \overline{X_{J}}.\overline{X_{I}}\bigg ]\\
\nonumber &=& \bigg(n.(n-1) n^{2(d-2)}+ n.1.n.(n-1).n^{2(d-3)}+ \dots + n.1 .n.1 \dots n . n.(n-1)\bigg)\\
\nonumber &. & (2\mu_{1}\overline{\mu^{\dagger}}+2\overline{\mu_{1}} \mu^{\dagger})\\
\nonumber & \leq & 4 \bigg[n^{2d-2}+n^{2d-3}+\dots +n^{2d-(d-1)}\bigg]\frac{\rho^{2}}{n^{2(d-1)}}\\
 & \leq & 4(d-2)^{2} \rho^{2} n^{2d-2}.
\end{eqnarray}
Here we used relations (\ref{m10}) and (\ref{s4}), and assumed that $d\geq 3$. 
\item
By according to multiplication of $ S_{1}^{2}\overline{S_{1}}^{2} $, we understand the indices of  $X_{i_{1}1\dots i_{d}}$ and $ X_{j_{1}1\dots j_{d}}$ that are placed together in the below form, they have property $ (i_{1},i_{3},\dots,i_{d}) <  (j_{1},i_{3},\dots,j_{d})$. For this, there is at least $ l \in \lbrace 1,3,\dots,d \rbrace$ that
\[
 i_{l} <  j_{l}, \quad 1 \leq j_{k} \leq n \quad \forall k=3,\dots,d, \quad k\neq l.
 \]
Thus, we have
\begin{eqnarray}\label{m16}
\nonumber\sum\limits_{\exists l, i_{l}< j_{l}}
& E & \bigg [ X_{I} X_{I} \overline{X_{J}} . \overline{X_{J}}+ X_{J} X_{J} \overline{X_{I}}. \overline{X_{I}}+ 4 X_{I} X_{J}\overline{X_{I}}.\overline{X_{J}}\bigg] \\
 \nonumber &=&\sum_{\substack { i_{1}< j_{1}}} 
E \bigg [ X_{I} X_{I} \overline{X_{J}} . \overline{X_{J}}+ X_{J} X_{J} \overline{X_{I}}. \overline{X_{I}} 
+ 4 X_{I} X_{J}\overline{X_{I}}.\overline{X_{J}}\bigg]\\
\nonumber &+& \dots 
 +  \sum_{\substack { i_{d}< j_{d}}} E \bigg [ X_{I} X_{I} \overline{X_{J}} . \overline{X_{J}}+ X_{J} X_{J} \overline{X_{I}}. \overline{X_{I}} + 4 X_{I} X_{J}\overline{X_{I}}.\overline{X_{J}}\bigg]\\
\nonumber &=& \Bigg(\binom{n}{2} n^{2d-4}+n \binom{n}{2} n^{2d-6} + n.n.\binom{n}{2} n^{2d-8}+\dots + n^{d-1}\binom{n}{2} \Bigg)\\
  \nonumber &.&\bigg(2\mu_{2}\overline{\mu_{2}}+4(\mu_{2}^{*})^{2}\bigg)\\
 \nonumber &\leq & 6 (d-2) n^{2(d-1)}\\ 
 &\leq & 6 (d-2)^{2} n^{2(d-1)} \rho^{2}.
 \end{eqnarray}
Here we used the fact $\binom{n}{2}\leq \frac{n^{2}}{2}$ and, relations (\ref{s1}) and (\ref{s2}). 
\item
We denote $  (k_{1},1,k_{3},\dots,k_{d}) $ with $K$. By according to the multiplication $S_{1}^{2}\overline{S_{1}}^{2}$, we understand the indices of  $X_{i_{1}1\dots i_{d}},X_{j_{1}1\dots j_{d}}$ and $ X_{k_{1}1\dots k_{d}}$ that are placed together in the below form, they have property $ J \neq K$  and $ I\neq J$ and $ I \neq  K $. Next, we must count the ways of choosing $X_{I}.X_{I}.\overline{X_{J}} . \overline{X_{K}}$ such that $ J \neq  K$  and $ I \neq  J $ and $ I \neq  K$. At first, we choose arbitrary $ (i_{1},1,i_{3},\dots,i_{d}) $ in $ n^{d-1} $ ways and put $ (i_{1} 1 \dots i_{d}) $ a constant $ d $-tuple. Then, we choose $ ({j_{1},\dots ,j_{d}}),(k_{1},\dots, k_{d}) $ in $ X_{J}. X_{K} $ such that $ K\neq J$ and $K \neq I $. For this, we define
\[
D= \lbrace X_{I}.X_{I}.\overline{X_{J}} . \overline{X_{K}} \quad\vert \quad    J\neq  K, \quad  I\neq J,\quad  I\neq K \rbrace.
\]
Also, we define  
\[
B= \lbrace X_{J}.X_{K} \quad\vert \quad J \neq  K\rbrace.
\]
and we define $ B_{1} , \dots, B_{d}$ in the below form
\begin{eqnarray*}
B_{1}&=&\lbrace X_{J}. X_{K} \quad\vert \quad j_{1}\neq k_{1}, \quad 1\leq j_{3},\dots ,j_{d}\leq n,\quad
1\leq k_{3},\dots ,k_{d} \leq n\rbrace,\\
\vdots\\
B_{d}&=&\lbrace X_{J} . X_{K} \quad\vert \quad   j_{s}= k_{s}, \quad \forall s=1,3\dots, d-1,\quad j_{d}\neq k_{d}\rbrace. 
\end{eqnarray*}
If $ n(B) $ be the number of members of set $ B $, it is clear that
\begin{eqnarray*}
n(B)=n(B_{1})+ n(B_{3})+\dots + n(B_{d})=n(n-1) n^{2d-4}+\dots +n^{d}(n-1). 
\end{eqnarray*}
In following, we define sets $ C,C_{1} , \dots, C_{d}$ 
\[
C=\lbrace  X_{I}.X_{K} \quad\vert \quad  K\neq I\rbrace.
\]
Notice that, the indices$ i_{1},\dots,i_{d} $ are constant in above. Now, we define
\begin{eqnarray*}
C_{1}&=&\lbrace X_{I}.X_{K}  \quad\vert \quad k_{1}\neq i_{1}, \quad 1\leq k_{3},\dots ,k_{d}\leq n\rbrace,\\
C_{2}&=&\lbrace X_{I}.X_{K}  \quad\vert \quad k_{1}=i_{1},\quad k_{3}\neq i_{3} \quad 1\leq k_{4},\dots ,k_{d}\leq n\rbrace,\\
\vdots\\
C_{d}&=&\lbrace X_{I}.X_{K} \quad\vert \quad   k_{s}= i_{s}, \quad \forall s=1,3,\dots, d-1,\quad k_{d}\neq i_{d}\rbrace. 
\end{eqnarray*}
It is clear that
\begin{eqnarray*}
n(C)=n(C_{1})+n(C_{3})+ \dots + n(C_{d})=(n-1) n^{d-2}+(n-1) n^{d-3}+ \dots + (n-1)
\end{eqnarray*}
Now, it is clear that
$ n(D)=n(B)-2n(c)$. Thus, we have
\begin{eqnarray}\label{m17}
\nonumber \sum_{\substack {\exists  l, j_{l}\neq k_{l}  \\ 
\exists  m, i_{m}\neq j_{m}\\
\exists r, i_{r}\neq k_{r}\\
}}& E& \bigg [ 2 X_{I} X_{I} \overline{X_{J}} .\overline{X_{K}} + 2 X_{J}X_{K} \overline{X_{I}}. \overline{X_{I}}+ 4 X_{I} X_{J} \overline{X_{I}}. \overline{X_{K}} + 4 X_{I}X_{K} \overline{X_{I}}. \overline{X_{J}} \bigg]\\
\nonumber &=& n^{d-1}\bigg [\bigg  (n(n-1) n^{2d-4}+\dots +n^{d}(n-1)\bigg)-2 \bigg( (n-1) n^{d-2}+(n-1) n^{d-3}+ \dots \\
\nonumber &+& (n-1) \bigg)\bigg].(2\mu_{2} \overline{\mu_{1}}^{2}+2 \mu_{1}^{2}\overline{\mu_{2}}+8 \mu_{2}^{*} \mu_{1}\overline{\mu_{1}})\\
\nonumber & \leq & n^{d-1}\bigg ( n^{2d-2}+n^{2d-3}+n^{2d-4}+\dots + n^{d+1} \bigg) (4 \frac{\rho^{2}}{n^{4(d-1)}} +4 \frac{\rho^{2}}{n^{4(d-1)}} +8 \frac{\rho^{2}}{n^{4(d-1)}})\\
\nonumber & \leq & (d-2) n^{2(d-1)} n^{(d-1)} (16 \frac{\rho^{2}}{n^{4(d-1)}}))\\
 & \leq &16 (d-2)^{2}\rho^{2} n^{2(d-1)}.
\end{eqnarray}
Here, we used relations (\ref{m10}), (\ref{s1}) and (\ref{s2}). 
\item
For convience, we denote $(l_{1},1,l_{3},\dots,l_{d}) $ with $L$. By according to the multiplication $ S_{1}^{2}\overline{S_{1}}^{2} $, we see the indices of  $X_{I},X_{J}, X_{K}$ and $ X_{L}$ that are placed together in the below form, they have property $ I \neq  J,  I \neq K, I\neq L , J\neq K, J \neq L$ and $ K\neq L$. Next, we must count the ways of choosing $X_{I}.X_{J}.\overline{X_{K}} . \overline{X_{L}}$ such that $I\neq J,  I \neq K,  I \neq L , J\neq K, J\neq L$ and $ K\neq L $. At first, we choose arbitrary $ (i_{1},1,i_{3},\dots,i_{d}) $ and $  (j_{1},1,j_{3},\dots,j_{d}) $ in $ X_{I}.X_{J}$ such that $ I \neq J $ in $n^{2d-3}(n-1)+ n^{2d-4}(n-1)+\dots + n^{d}(n-1)  $ ways. Then, we choose $ (k_{1},1,k_{3},\dots,k_{d}) $ and $  (l_{1},1,l_{3},\dots,l_{d}) $ such that $ K\neq L, K\neq I$ and $ K\neq J $. For this, we define the below sets.
\begin{eqnarray*}
B &=& \lbrace X_{I}.X_{J}.\overline{X_{K}}.\overline{X_{L}} \quad \vert \quad  I \neq  J,\quad  I \neq K,\quad  I \neq L,\quad J\neq K,\quad J\neq L, \quad K\neq L\rbrace,\\
C &=& \lbrace  X_{I}.X_{J} \quad \vert \quad  I \neq  J\rbrace,\\
D &=& \lbrace  X_{K}.X_{L} \quad \vert \quad K \neq  L\rbrace,\\
E &=& \lbrace  X_{K}.X_{L} \quad \vert \quad K \neq  L, \quad K\neq I\rbrace, \\
F &=& \lbrace  X_{K}.X_{L} \quad \vert \quad K \neq  L, \quad K\neq J\rbrace, 
\end{eqnarray*}
since $ n(E)=n(F) $, 
\begin{eqnarray*}
n(B)&=&n(c).\bigg(n(D)-2n(E)-2n(F)+2\bigg)\\
&=&\Bigg(n^{2d-3}(n-1)+ n^{2d-4}(n-1)+\dots + n^{d}(n-1)\Bigg) \\
\nonumber &. &\bigg[\bigg(n^{2d-3}(n-1)+ n^{2d-4}(n-1)+\dots + n^{d}(n-1)\bigg) \\
\nonumber &-& 4 \bigg( (n-1)n^{d-2}+(n-1)n^{d-3}+\dots +(n-1)\bigg)+2\bigg] 
\end{eqnarray*}
Thus, we have
\begin{eqnarray}\label{m18}
\nonumber \sum _{\substack { 
\exists r, i_{r}\neq j_{r}\\
\exists s, i_{s}\neq k_{s}\\
\exists  t, i_{t}\neq l_{t}\\
\exists  r^{\prime}, j_{r^{\prime}}\neq k_{r^{\prime}}\\
\exists  s^{\prime}, j_{s^{\prime}}\neq l_{s^{\prime}}\\
\exists  t^{\prime}, k_{t^{\prime}}\neq l_{t^{\prime}}\\
 }}
 &E& \bigg [X_{i_{1} 1 i_{3} \dots i_{d}} X_{j_{1} 1 j_{3} \dots j_{d}}\overline{X_{k_{1} 1 k_{3} \dots k_{d}}}. \overline{X_{l_{1} 1 l_{3} \dots l_{d}}}\bigg ] \\
\nonumber & = & \bigg(n^{2d-3}(n-1)+ n^{2d-4}(n-1)+\dots + n^{d}(n-1)\bigg) \\
\nonumber &. & \bigg[\bigg(n^{2d-3}(n-1)+ n^{2d-4}(n-1)+\dots + n^{d}(n-1)\bigg) \\
\nonumber &-& 4 \bigg( (n-1)n^{d-2}+(n-1)n^{d-3}+\dots +(n-1)\bigg)+2\bigg] \vert \mu_{1}\vert^{4} \\
\nonumber & \leq & \bigg (n^{2d-2}+\dots + n^{d+1}\bigg) \bigg[\bigg(n^{2d-2}+\dots + n^{d+1}\bigg) +2\bigg] \vert \mu_{1}\vert^{4}\\
\nonumber & \leq & (d-2) n^{2d-2}\bigg( (d-2) n^{2d-2}+2\bigg) \vert \mu_{1}\vert^{4} \\
\nonumber & \leq & \bigg( (d-2)^{2} n^{4(d-1)} +2 (d-2)^{2} n^{2d-2}\bigg) \frac{\rho^{4}}{n^{8(d-1)}}\\
\nonumber & \leq & 3 (d-2)^{2} n^{4(d-1)}\frac{\rho^{4}}{n^{8(d-1)}}\\
 & \leq & 3 (d-2)^{2}\rho^{2} n^{2(d-1)}.
\end{eqnarray}
Here, we used the fact $ \frac{\rho^{2}}{n^{2(d-1)}}\leq 1$. Now, by relations (\ref{m21}), (\ref{m15}), (\ref{m16}), (\ref{m17}) and (\ref{m18}) we obtain
\begin{eqnarray*}
\nonumber E[S_{1}^{2}\overline{S_{1}}^{2}] &= & \sum \limits_{i_{1},i_{3},\dots , i_{d}=1}^{n} E[ X_{I}^{2} \overline{X_{I}}^{2}] \\
 \nonumber &+&  \sum\limits_{\exists  l, i_{l}\neq j_{l}} 
 E  \bigg[ X_{I} X_{J} \overline{X_{J}}. \overline{X_{J}} +  X_{J} X_{I}  \overline{X_{J}} . \overline{X_{J}} + X_{J} X_{J} \overline{X_{I}}. \overline{X_{J}} + X_{J} X_{J} \overline{X_{J}} .\overline{X_{I}}\bigg] \\
\nonumber &+&\sum\limits_{\exists l, i_{l}< j_{l}}
 E  \bigg [X_{I} X_{I} \overline{X_{J}} . \overline{X_{J}}+ X_{J} X_{J} \overline{X_{I}}. \overline{X_{I}}+ 4 X_{I} X_{J}\overline{X_{I}}.\overline{X_{J}}\bigg] \\
\nonumber &+ & \sum_{\substack {\exists  l, j_{l}\neq k_{l}  \\ 
\exists  m, i_{m}\neq j_{m}\\
\exists  r, i_{r}\neq k_{r}\\
}}  E \bigg [ 2 X_{I} X_{I} \overline{X_{J}} .\overline{X_{K}} + 2 X_{J}X_{K} \overline{X_{I}}. \overline{X_{I}}+ 4 X_{I} X_{J} \overline{X_{I}}. \overline{X_{K}} + 4 X_{I}X_{K} \overline{X_{I}}. \overline{X_{J}} \bigg] \\
\nonumber &+& \sum _{\substack { 
\exists  r, i_{r}\neq j_{r}\\
\exists s, i_{s}\neq k_{s}\\
\exists t, i_{t}\neq l_{t}\\
\exists  r^{\prime}, j_{r^{\prime}}\neq k_{r^{\prime}}\\
\exists  s^{\prime}, j_{s^{\prime}}\neq l_{s^{\prime}}\\
\exists  t^{\prime}, k_{t^{\prime}}\neq l_{t^{\prime}}\\
 }}
 E \bigg [4X_{I} X_{J}\overline{X_{K}}. \overline{X_{L}}\bigg ] \\
 & \leq & 34 (d-2)^{2} \rho^{2} n^{2(d-1)}.
\end{eqnarray*}
Moreover, relation (\ref{m13}) and Chebyshev's inequality allow us to conclude that
\begin{eqnarray*}
\nonumber P(\bigg \vert \frac{S}{n^{2(d-1)}}\bigg \vert > \varepsilon ) &\leq & \frac{E[S \overline{S}]}{\varepsilon^{2} n^{4(d-1)}}\\
\nonumber & \leq & \frac{n E[S_{1}^{2}\overline{S_{1}}^{2}] + n^{2}\bigg \vert E[S_{1}^{2}-n^{d-1}\xi] \bigg \vert^{2}}{\varepsilon^{2} n^{4(d-1)}} \\
\nonumber & \leq & \frac{34 n (d-2)^{2}n^{2(d-1)} \rho^{2}+ n^{2}(\rho + (d-2))^{2}}{\varepsilon^{2} n^{4(d-1)}}\\
\nonumber & \leq & \frac{38 (d-2)^{2}n^{2d-1} \rho^{2} }{\varepsilon^{2} n^{4(d-1)}}\\
\nonumber & = & \frac{38(d-2)^{2}\rho^{2}}{\varepsilon^{2} n^{2d-3}}.
\end{eqnarray*}
Let $\varepsilon = n^{-\phi}$,
\begin{eqnarray*}
\nonumber P\left(\bigg \vert \frac{S}{n^{2(d-1)}}\bigg \vert > \varepsilon \right) \leq \frac{38 (d-2)^{2}\rho^{2}}{n^{-2\phi} n^{2d-3} }.
\end{eqnarray*}
If $0 < \phi <\frac{2d-3}{2}$, then by relation (\ref{m9}),
\begin{eqnarray*}
\nonumber P( \vert D_{2}-\xi \vert \leq n^{- \phi}) \geq 1- \frac{\rho}{n^{d-1}} - \frac{38 (d-2)^{2}\rho^{2}}{n^{-2\phi+2d-3}} = 1- o(1).
\end{eqnarray*}
For this, $d$ must be very small compared to $n$ or $d \in \mathcal{O}(polylog(n))  $, and this completes the proof.
\end{proof}
According to Lemma \ref{m22}, we can approximate $D_{2} $ by $\xi$, the quasi-variance of $\mathcal{D}$. The following lemma will be used in the proof of Lemma \ref{01}. 
\begin{lemma}\label{m24}
Let $\lbrace X_{i_{1}\dots i_{d}}\rbrace_{i_{1},\dots,i_{d}=1}^{n}$ be a sequence of i.i.d random variables with distribution $\mathcal{D}_{0}$. Then, there exists an absolute constant $\eta > 0$ such that 
\[
E \Bigg[ \Bigg \vert \frac{\sum\limits_{i_{1},i_{3},\dots,i_{d}=1}^{n}X_{i_{1} \dots i_{d}}}{\sqrt{n^{d-1}}} \Bigg\vert ^{3} \Bigg] \leq \eta (1+\frac{\rho}{\sqrt{n^{d-1}}}).
\]
\end{lemma}
\begin{proof}
First, suppose that
\[
\sigma_{1}=\sqrt{E_{X\thicksim \mathcal{D}}[(Re(X))^{2}]}, \quad \sigma_{2}=\sqrt{E_{X\thicksim \mathcal{D}}[(Im(X))^{2}]}.
\]
Also, let
\[\rho_{1}=E_{X\thicksim \mathcal{D}}[\vert Re(X)\vert^{3}], \quad \rho_{2}=E_{X\thicksim \mathcal{D}}[\vert Im(X)\vert^{3}]
\]
and
\[
x_{i_{1} i_{2} \dots i_{d}}=Re(X_{i_{1} i_{2} \dots i_{d}}),\quad y_{i_{1} i_{2} \dots i_{d}}=Im(X_{i_{1} i_{2} \dots i_{d}}).
\]
Moreover, define
\[
R_{m}=\sum\limits_{i_{1},i_{3},\dots,i_{d}=1}^{m}x_{i_{1} i_{2} i_{3}\dots i_{d}}, \quad T_{m}=\sum\limits_{i_{1},i_{3},\dots,i_{d}=1}^{m}y_{i_{1} i_{2} i_{3}\dots i_{d}}.
\]
Since for each $i_{1},\dots,i_{d}$, $X_{i_{1}\dots i_{d}}$ are i.i.d variables with distribution $\mathcal{D}$, we can write
\begin{eqnarray}\label{m36}
\nonumber E [\vert R_{2k} \vert^{3}]&=& E[\vert R_{k}+(R_{2k}-R_{k})\vert^{3}]\\
\nonumber &\leq & E[(\vert R_{k}\vert + \vert R_{2k}-R_{k}\vert)^{3}]\\
\nonumber &=& E[\vert R_{k} \vert^{3}]+ E[\vert R_{2k}-R_{k}\vert^{3}] + 3E[\vert R_{k} \vert^{2} \vert R_{2k}-R_{k} \vert] + 3 E[\vert R_{k}\vert \vert R_{2k}-R_{k}\vert^{2}] \\
\nonumber &\leq & E[\vert R_{k} \vert^{3}] + ((2k)^{d-1}-k^{d-1})\rho_{1}\\
\nonumber &+& 3k^{d-1}\sqrt{(2k)^{d-1}-k^{d-1}}\sigma_{1}^{3}+ 3\sqrt{(2k)^{d-1}}((2k)^{d-1}-k^{d-1})\sigma_{1}^{3} \\ 
\nonumber &=& E[\vert R_{k} \vert^{3}] + ((2k)^{d-1}-k^{d-1})\rho_{1} + 3 k^{d-1}\sqrt{(2k)^{d-1}-k^{d-1}}\sigma_{1}^{3}\\
\nonumber & +& 3\sqrt{(2k)^{d-1}}((2k)^{d-1}-k^{d-1})\sigma_{1}^{3}. \\
\end{eqnarray}
Also,
\begin{eqnarray}\label{m37}
\nonumber E[\vert R_{2k+1}\vert^{3}]&\leq & E[ \vert R_{2k}\vert ^{3}]+ ((2k+1)^{d-1}-(2k)^{d-1}) E[ \vert x_{2k+1 \dots 2k+1}\vert^{3}] \\
\nonumber &+& 3 ((2k+1)^{d-1}-(2k)^{d-1}) E[\vert R_{2k} \vert^{2}] \sqrt{E[\vert x_{2k+1\dots 2k+1} \vert^{2}]}\\
\nonumber &+& 3 ((2k+1)^{d-1}-(2k)^{d-1})\sqrt{E[\vert R_{2k}\vert^{2}]} E[\vert x_{2k+1\dots 2k+1} \vert^{2}]\\
\nonumber &=& E[\vert R_{2k} \vert^{3}] + ((2k+1)^{d-1}-(2k)^{d-1}) \rho_{1}\\
\nonumber &+& 3 ((2k+1)^{d-1}-(2k)^{d-1})(2k)^{d-1}\sigma_{1}^{3}\\
 &+& 3 ((2k+1)^{d-1}-(2k)^{d-1}) \sqrt{(2k)^{d-1}}\sigma_{1}^{3}.
\end{eqnarray}
Now, it follows from relations (\ref{m36}) and (\ref{m37}) and induction that
\begin{eqnarray}\label{mn38}
E[\vert R_{n}\vert^{3}]\leq C^{\prime}(n^{d-1}\rho_{1}+n^{\frac{3(d-1)}{2}}\sigma_{1}^{3}),
\end{eqnarray}
for some $C^{\prime}$. Similarly, we obtain
\begin{eqnarray}\label{m388}
E[\vert T_{n}\vert^{3}]\leq C^{\prime}(n^{d-1}\rho_{2}+n^{\frac{3(d-1)}{2}}\sigma_{2}^{3}).
\end{eqnarray}
Also, we obtain the following relations for $0\leq k \leq n$.
\begin{eqnarray}\label{m399}
\nonumber E[\vert R_{n}\vert^{2} \vert T_{n}\vert]&=& E [ \vert R_{k} + (R_{n} - R_{k})\vert^{2}\vert T_{k}+ (T_{n} - T_{k}) \vert ]\\
\nonumber &\leq & E[\vert R_{k}\vert^{2} \vert T_{k}\vert ]+ \vert R_{n} - R_{k}\vert^{2}\vert T_{n} -T_{k}\vert] +3 k^{d-1}\sqrt{n^{d-1} - k^{d-1}}\sigma_{1}^{2}\sigma_{2}\\
 &+& 3\sqrt{k^{d-1}}(n^{d-1}-k^{d-1}) \sigma_{1}^{2}\sigma_{2}.
\end{eqnarray}
Relation (\ref{m399}) allows us to conclude that for some $ C^{\prime\prime}$, 
\begin{eqnarray}\label{m40}
E[\vert R_{n}\vert^{2}\vert T_{n}\vert ]\leq n^{d-1}E[\vert x_{1}^{2}y_{1}\vert] + C^{\prime\prime} n^{\frac{3(d-1)}{2}}\sigma_{1}^{2}\sigma_{2}.
\end{eqnarray}
Similarly,
\begin{eqnarray}\label{m41}
E[\vert R_{n}\vert \vert T_{n}\vert^{2} ]\leq n^{d-1}E[\vert x_{1}y_{1}^{2}\vert] + C^{\prime\prime} n^{\frac{3(d-1)}{2}}\sigma_{1}\sigma_{2}^{2}.
\end{eqnarray}
Thus, it follows from relations (\ref{mn38}), (\ref{m388}), (\ref{m40}) and (\ref{m41}) that
\begin{eqnarray}
\nonumber E \Bigg[ \Bigg \vert \frac{\sum\limits_{i_{1},i_{3},\dots,i_{d}=1}^{n}X_{i_{1} \dots i_{d}}}{\sqrt{n^{d-1}}}  \Bigg\vert ^{3} \Bigg]&=& n^{\frac{-3(d-1)}{2}} E[\vert R_{n}+iT_{n}\vert^{3}] \\
\nonumber &\leq & n^{\frac{-3(d-1)}{2}} E[\vert R_{n}\vert^{3}+\vert T_{n}\vert^{3}+3\vert R_{n}\vert^{2}\vert T_{n}\vert +3\vert R_{n} \vert \vert T_{n}\vert^{2}]\\
\nonumber & \leq & n^{\frac{-3(d-1)}{2}} \bigg( C^{\prime} n^{d-1}(\rho_{1}+\rho_{2})+ C^{\prime} n^{\frac{3(d-1)}{2}}(\sigma_{1}^{3}+\sigma_{2}^{3})\\
\nonumber &+& 3 C^{\prime\prime} n^{\frac{3(d-1)}{2}} \sigma_{1}\sigma_{2}(\sigma_{1}+\sigma_{2})+3n^{d-1}E[\vert x_{1}\vert^{2} \vert y_{1}\vert +\vert x_{1}\vert \vert y_{1}\vert^{2}] \bigg).
\end{eqnarray}
We know that $\rho_{1},\rho_{2}\leq \rho$, $\sigma_{1},\sigma_{2}\leq1$. Also, for all $x,y> 0$,
$ x^{2}y+xy^{2} \leq x^{3}+y^{3}$. Thus, for some constant $ \eta >0 $ we obtain
\begin{eqnarray}
\nonumber E \Bigg[ \Bigg \vert \frac{\sum\limits_{i_{1},i_{3},\dots,i_{d}=1}^{n}X_{i_{1} \dots i_{d}}}{\sqrt{n^{d-1}}}  \Bigg\vert ^{3} \Bigg] &\leq & n^{\frac{-3(d-1)}{2}}\bigg [2C^{\prime} n^{d-1}\rho + 2 C^{\prime} n^{\frac{3(d-1)}{2}}+6 C^{\prime\prime}n^{\frac{3(d-1)}{2}} + 3 n^{d-1} E[ \vert x_{1}\vert^{3}+\vert y_{1}\vert^{3}]\bigg]\\
\nonumber &\leq & \eta (1+\frac{\rho}{\sqrt{n^{d-1}}}).
\end{eqnarray}
\end{proof}
  In what follows, we show that $ \vert D_{k} \vert$ is small for $k\geq 3$. 
\begin{lemma}\label{01}
Let $ \triangle$ be an arbitrary positive constant, and $ \triangle < \frac{3d-5}{6}$. Then for any $k\geq 3$,
\begin{eqnarray*}
P( \vert D_{k} \vert \leq n^{-\triangle k})=1-o(1).
\end{eqnarray*}
\end{lemma}
\begin{proof}
This is equivalent to 
\begin{eqnarray*}
P(\exists k\geq 3, \quad \vert D_{k} \vert > n^{-\triangle k})=o(1).
\end{eqnarray*}
Now,
\begin{eqnarray}
\nonumber P(\exists k\geq 3, \quad  \vert D_{k} \vert > n^{-\triangle k}) &=& P \bigg (\exists k\geq 3, \quad \dfrac{1}{n^{\frac{k(d-1)}{2}}}\sum\limits_{j=1}^{n} \vert C_{j} \vert^{k} >n^{-\triangle k} \bigg)\\
\nonumber &=& P \bigg (\exists k\geq 3, \quad \sum\limits_{j=1}^{n} \vert C_{j} \vert^{k} >n^{k(-\triangle + \frac{d-1}{2})} \bigg)\\
\nonumber &=& P \bigg (\exists k\geq 3, \quad \bigg (\sum\limits_{j=1}^{n} \vert C_{j} \vert^{k}\bigg)^{\frac{1}{k}} >n^{\frac{d-1}{2}-\triangle} \bigg)\\
\nonumber &\leq &P \bigg( \bigg (\sum\limits_{j=1}^{n} \vert C_{j} \vert^{3}\bigg)^{\frac{1}{3}} >n^{\frac{d-1}{2}-\triangle} \bigg) \\
\nonumber &\leq & P \bigg (\sum\limits_{j=1}^{n} \vert C_{j} \vert^{3} >n^{\frac{3(d-1)}{2}-3\triangle} \bigg), 
\end{eqnarray}
where the last step follows from the properties of the $L^{p}$ norm. Also, Markov's inequality allows us to conclude that
\begin{eqnarray*}
\nonumber P(\exists k\geq 3, \quad  \vert D_{k} \vert > n^{-\triangle k}) &\leq & P  \bigg (\sum\limits_{j=1}^{n} \vert C_{j} \vert^{3} > n^{\frac{3(d-1)}{2}-3\triangle} \bigg)\\
\nonumber & \leq & \frac{E \bigg [\sum\limits_{j=1}^{n} \vert C_{j} \vert^{3} \bigg] }{n^{\frac{3(d-1)}{2}-3\triangle}}\\
& = & \frac{n E[ \vert C_{j} \vert^{3} ]}{n^{\frac{3(d-1)}{2}-3\triangle}}.
\end{eqnarray*}
By Lemma \ref{m24}, we know that there exists a positive constant $ \eta $ such that 
\[
E [ \vert C_{j} \vert^{3} ]\leq \eta (1+\frac{\rho}{\sqrt{n^{d-1}}}), \quad \forall j \in [n].
\]
Now,
\begin{eqnarray*}
\nonumber P(\exists k\geq 3, \quad  \vert D_{k} \vert > n^{-\triangle k}) & \leq & \frac{n E[ \vert C_{j} \vert^{3} ]}{n^{\frac{3(d-1)}{2}-3\triangle}}\\
\nonumber & \leq & \dfrac{n . \eta (1+\frac{\rho}{\sqrt{n^{d-1}}})}{n^{\frac{3(d-1)}{2}-3\triangle}} \\\\
\nonumber &=& n^{1-\frac{3(d-1)}{2}+3\triangle}. \eta (1+\frac{\rho}{\sqrt{n^{d-1}}}).
\end{eqnarray*}
If $ \triangle < \frac{3d-5}{6}$, then the right hand is $ o(1)$, and this completes the proof.
\end{proof}
The following lemma is proved using the definitions of elementary symmetric polynomials and power sums.
\begin{lemma}[\cite{pin}]\label{m27}
Let $x_{1},\dots,x_{n}$ be arbitrary variables then for any $m \in [n]$,
\[
e_{m}(n)=\frac{1}{m} \sum\limits_{k=0}^{m-1}(-1)^{k} e_{m-k-1}(n) S_{k+1}(n).
\]
\end{lemma}
Using Lemma \ref{m27} and the definitions of $V_{k},S_{k}$ and $e_{k} $, we observe that $V_{k}$ satisfies the following relation for $k\geq 2$.  
\begin{eqnarray}
\nonumber V_{k}=\dfrac{V_{k-1}V_{1}-V_{k-2}D_{2}+\sum\limits_{i=2}^{k-1}(-1)^{i}V_{k-1}D_{i+1}}{k}.
\end{eqnarray}
Now, by Lemma \ref{m22} and Lemma \ref{01}, we consider an asymptotic approximation $V^{\prime}_{k}$ of $V_{k}$ of the following form.
\begin{equation}\label{ms28}
V^{\prime}_{k}=\begin{cases}
	1, & k= 0\\
V_{1}, &   k=1\\
\frac{V^{\prime}_{k-1} V^{\prime}_{1}- V^{\prime}_{k-2}\xi}{k},  &  k \geq 2.
\end{cases}
\end{equation} 

\section{Approximation of $ \sum\limits_{k=0}^{t} V_{k}z^{k} $}
In this section, similar to the case of matrices, first we use probabilists' Hermite polynomials to show that 
$ \sum\limits_{k=0}^{\infty} V^{\prime}_{k} z^{k}=e^{V_{1}z- \frac{\xi z^{2}}{2}}$, where $V_{1}$ is the normalization of the entries of tensor $A$. Then, we show that with high probability, $ \vert e^{V_{1}z- \frac{\xi z^{2}}{2}} \vert \geq n^{-\gamma}$. Also, we prove that $ \sum\limits_{k=0}^{t} V_{k}z^{k} $ can be approximated by $ \sum\limits_{k=0}^{t} V^{\prime}_{k}z^{k} $.
\begin{definition}[\cite{pin}]
Polynomials in the following form are called Probabilists' Hermite polynomials
\[
H_{e_{n}}(x)=(-1)^{n} e^{\frac{x^{2}}{2}} \frac{d^{n}}{dx^{n}}e^{\frac{-x^{2}}{2}},
\]
for $n \in \mathbb{N}$. Also, $ H_{e_{n}}(x) $ satisfies 
\begin{equation*}
H_{e_{n}}(x)=\begin{cases}
1,     &   n= 0,\\
x, &   n=1,\\
x H_{e_{n-1}}(x) - (n-1) H_{e_{n-2}}(x) ,  & n \geq 2.
\end{cases}
\end{equation*}
Moreover, $h_{n}(x)$ is defined by $\frac{1}{n!} H_{e_{n}}(x)$, and satisfies
\begin{equation}\label{m28}
h_{n}(x)=\begin{cases}
1,     &   n= 0,\\
x, &   n=1,\\
\frac{x h_{n-1}(x) -  h_{n-2}(x)}{n} ,  & n \geq 2.
\end{cases}
\end{equation}
\end{definition}
\begin{lemma}[\cite{pin}]\label{m29}
For any $n \in \mathbb{N}$ and any $x \in \mathbb{C}$, 
\[
\vert h_{n}(x) \vert \leq \max (1,\vert x \vert)^{n}(\frac{n}{e^{2}})^{\frac{-n}{2}}.
\]
\end{lemma}
In what follows, similar to the case of matrices, we use relations (\ref{m28}) and (\ref{ms28}) to write
\begin{eqnarray}\label{m30}
 \left \{
\begin{array}{ll}
V^{\prime}_{k}=\frac{V_{1}^{k}}{k!},   \quad \textit{if} & \xi =0,\\
\\
V^{\prime}_{k}=\xi^{\frac{k}{2}}h_{k}(\frac{V_{1}}{\sqrt{\xi}}),    \quad  \textit{if} & \xi \neq 0.
\end{array} \right.
\end{eqnarray}
Here, $V_{1}$ is the normalization of the entries of tensor $A$. The following lemma will be used in the proofs of lemmas that will appear later.
\begin{lemma}\label{m31}
If $\theta(n)=\omega(1)$, then with probability $ 1-o(1) $
\[
\vert V_{1} \vert \leq \theta.
\]
\end{lemma}
\begin{proof}
The proof of this lemma is similar to that of its matrix case, except that here $V_{1}$ is the normalization of the entries of tensor $A$. Thus, we  refer the reader to the proof of lemma $6.2$ in \cite{pin} for the details.
\end{proof}
\begin{lemma}\label{m311}
If $k \in \mathbb{N}$ is arbitrary, then
\[
\vert V^{\prime}_{k} \vert \leq \max (1, \vert V_{1} \vert^{k}) \bigg(\frac{k}{e^{2}}\bigg)^{\frac{-k}{2}}.
\]
\end{lemma}
\begin{proof}
The proof of this lemma can be completed using relations (\ref{m30}) and (\ref{m1}) and Lemma \ref{m29}. The proof is similar to that of its matrix case, except that here $V_{1}$ is the normalization of the entries of tensor $A$, and $V^{\prime}_{k}$ is an asymptotic approximation of $V_{k}$ associated to $A$. Therefore, we refer the reader to the proof of lemma $6.3$ in \cite{pin} for the details.
\end{proof}
The following lemma presents a result for tensors, which is similar to a well-known statement from matrix theory.
\begin{lemma}\label{mm00}
Suppose that $\theta:=o(\ln^{\frac{1}{4}}n)$ is a function of $n$ such that $\theta \geq 1$ and $\vert V_{1} \vert \leq \theta$. For any constant $\tau > 0$ and for sufficiently large $n$,
\[
\vert V^{\prime}_{k} \vert \leq n^{\tau} k^{\frac{-k}{4}}.
\]
Also, it is true that
\[
\vert V^{\prime}_{k} \vert \leq e^{2\theta^{2}}.
\]
\end{lemma}
\begin{proof}
The proof of this lemma can be completed by Lemma \ref{m311}. The proof is similar to that of its matrix case, except that  $V^{\prime}_{k}$ is an asymptotic approximation of $V_{k}$ associated to $A$. Thus, we refer the reader to the proof of lemma $6.4$ in \cite{pin} for the details.
\end{proof}
In below, we present two well-known expansion formulas as a generating function.
\begin{lemma}[ \cite{boy}, \cite{beta}]
For any $ z,t \in \mathbb{C}$,
\begin{eqnarray*}\label{m300}
 \left \{
\begin{array}{ll}
\sum\limits_{k=0}^{\infty} \frac{z^{k}}{k!}=e^{z},\\
\\
\sum\limits_{k=0}^{\infty} \frac{H_{e_{k}}(z) t^{k}}{k!}=e^{zt-\frac{t^{2}}{2}}.
\end{array} \right.
\end{eqnarray*}
\end{lemma}
Using relation \ref{m300}, similar to the case of matrices, we can write 
\begin{eqnarray}\label{m33}
\sum\limits_{k=0}^{\infty} V^{\prime}_{k} z^{k}= e^{V_{1}z-\frac{\xi z^{2}}{2}},
\end{eqnarray}
where $V_{1}$ is the normalization of the entries of tensor $A$, and $\xi$ is the quasi-variance of $\mathcal{D}$. 

 In the following lemma, we prove that we can consider $\sum\limits_{k=0}^{\infty} V^{\prime}_{k} z^{k}$ as $\sum\limits_{k=0}^{t} V^{\prime}_{k} z^{k}$ with a small additive error.
\begin{lemma}\label{17}
Considering relations (\ref{aa}),
\[
P\bigg[\bigg \vert \sum\limits_{k=t+1}^{\infty} V^{\prime}_{k} z^{k} \bigg \vert = n^{-\omega(1)}\bigg]=1-o(1).
\]
\end{lemma}
\begin{proof}
Let $\theta(n)=\ln \ln n$. Since $\theta(n)=\omega(1)$, by Lemma \ref{m31} with high probability, $ \vert V_{1} \vert \leq \theta$. Now, by Lemma \ref{mm00},
\begin{eqnarray} \label{m38}
\bigg \vert \sum\limits_{k=t+1}^{\infty} V^{\prime}_{k} z^{k} \bigg \vert\leq n^{\tau} \sum\limits_{k=t+1}^{\infty} k^{\frac{-k}{4}} \vert z \vert^{k}.
\end{eqnarray}
Since $\vert z \vert^{8} \leq \frac{\ln n}{(d-1)^{6}}$ and $t=\ln n+\ln (\frac{1}{\varepsilon})$, for large $n$ and $d$,
\begin{eqnarray}\label{m39}
\nonumber \frac{(k+1)^{\frac{-(k+1)}{4} }\vert z \vert^{k+1}}{(k)^{\frac{-k}{4}} \vert z \vert^{k}}&=& \frac{\vert z \vert}{(k+1)^{\frac{1}{4}}}\bigg(1+\frac{1}{k}\bigg)^{\frac{-k}{4}}\\
\nonumber &\leq & \frac{\vert z \vert}{t^{\frac{1}{4}}}\\
\nonumber &\leq & \dfrac{(\ln n)^{\frac{1}{8}}}{(d-1)^{\frac{6}{8}} (\ln (\frac{n}{\varepsilon}))^{\frac{1}{4}}}\\
&\leq & \frac{1}{2(d-1)^{\frac{6}{8}}}.
\end{eqnarray}
Now, using relations (\ref{m38}) and (\ref{m39}) we obtain
\begin{eqnarray} 
\nonumber \bigg \vert \sum\limits_{k=t+1}^{\infty} V^{\prime}_{k} z^{k} \bigg \vert &\leq & n^{\tau}  \bigg( (t+1)^{\frac{-(t+1)}{4}} \vert z \vert^{t+1}+ (t+2)^{\frac{-(t+2)}{4}} \vert z \vert^{t+2}+ \dots \bigg)\\
\nonumber &\leq & n^{\tau}  \bigg( \frac{1}{2(d-1)^{\frac{6}{8}}}. t^{\frac{-t}{4}}+\frac{1}{4(d-1)^{\frac{12}{8}}}. t^{\frac{-t}{4}}+\dots \bigg)\vert z \vert^{t}\\
\nonumber &\leq & n^{\tau} t^{\frac{-t}{4}}\vert z\vert^{t} \frac{1}{2(d-1)^{\frac{6}{8}}-1}\\
\nonumber &=& n^{\tau} \frac{1}{\ln (\frac{n}{\epsilon})^{\frac{1}{4} \ln (\frac{n}{\varepsilon})}}. 
(\frac{\ln n}{(d-1)^{6}})^{\frac{1}{8}\ln \frac{n}{\varepsilon}}
.\frac{1}{2(d-1)^{\frac{6}{8}}-1},
\end{eqnarray}
and for large $n$ and $d$, this is very small, namely, $n^{-\omega(1)}$. 
\end{proof}
In what follows, we prove that $\bigg\vert e^{V_{1}z-\frac{\xi z^{2}}{2}} \bigg\vert \geq n^{-\gamma}$ with high probability. 
\begin{lemma}\label{18}
Considering relations (\ref{aa}), with probability $ 1-o(1) $
\[
 \bigg\vert e^{V_{1}z-\frac{\xi z^{2}}{2}} \bigg\vert \geq n^{-\gamma} 
\]
\end{lemma}
\begin{proof}
First,
\begin{eqnarray}
\nonumber P\bigg[ \bigg\vert e^{V_{1}z-\frac{\xi z^{2}}{2}} \bigg\vert < n^{-\gamma} \bigg] &=& P\bigg [e^{Re(V_{1}z-\frac{\xi z^{2}}{2})} < n^{-\gamma} \bigg]\\
\nonumber &=&P \bigg[ Re \bigg(V_{1}z-\frac{\xi z^{2}}{2}\bigg)< -\gamma \ln n\bigg]\\
\nonumber & =& P\bigg [ Re \bigg(V_{1}\frac{z}{\vert z \vert}\bigg) <\frac{-\gamma \ln n}{\vert z \vert}+\frac{Re(\xi z^{2})}{2 \vert z \vert} \bigg].
\end{eqnarray}   
Since $\vert z \vert^{8}\leq \frac{\ln n}{(d-1)^{6}} $ and $\gamma <\frac{1}{2}<1$, 
\[
\frac{Re(\xi z^{2})}{2\vert z \vert}\leq \bigg \vert \frac{\xi z^{2}}{2  z } \bigg \vert\leq\frac{\gamma  \ln n}{2(d-1)^{\frac{12}{8}}\vert z \vert}.
\]
Also, for large $n$ and $d$ we obtain $ \frac{\gamma \ln n }{2(d-1)^{\frac{12}{8}}\vert z \vert} =\omega(1)$, and we can use Lemma \ref{m31}. Thus, 
\[
P\bigg[ \bigg\vert e^{V_{1}z-\frac{\xi z^{2}}{2}} \bigg\vert < n^{-\gamma} \bigg] \leq P\bigg[ \vert V_{1} \vert > \frac{\gamma \ln n}{2 (d-1)^{\frac{12}{8}} \vert z \vert}\bigg ]=o(1).
\]
Notice that here, we need $ n=exp[\Omega(\gamma^{\frac{-1}{1-c}})] $.
\end{proof}

In the following lemma, we find an upper bound for $\vert V^{\prime}_{k}-V_{k} \vert$. 
\begin{lemma}\label{41}
Let $\theta(n)= \ln \ln n$, and $\nu< \frac{1}{8}$ be a positive constant. Then, there exists a constant $ n_{k} $ such that with high probability, 
\[
\varepsilon_{k}:=\vert V^{\prime}_{k}-V_{k} \vert \leq n^{-\nu}k^{-\nu k}, \quad \forall n> n_{k},
\]
for all $k$, $ 0\leq k \leq n$.
\end{lemma}
\begin{proof}
The proof is similar to that of its matrix case, except that $V^{\prime}_{k}$ is an asymptotic approximation of $V_{k}$ associated to $A$, and also $D_{k}$ is defined in relation (\ref{aa}) associated to $A$, and in this case, $ \frac{1}{8}< \triangle < \frac{3d-5}{6}$. Thus, we refer the reader to the proof of lemma $7.1$ in \cite{pin} for the details.
\end{proof}
\begin{lemma}\label{42}
Considering relations (\ref{aa}), 
\[
P \bigg[\bigg\vert \sum\limits_{k=0}^{t} V_{k}z^{k}-\sum\limits_{k=0}^{t} V_{k}^{\prime}z^{k}\bigg\vert = \mathcal{O}(n^{c-\nu})\bigg]=1-o(1).
\]
\end{lemma} 
\begin{proof}
The proof is similar to that of its matrix case, except that $V^{\prime}_{k}$ is an asymptotic approximation of $V_{k}$ associated with tensor $A$. Thus, we refer the reader to the proof of lemma $7.2$ in \cite{pin} for the details. 
\end{proof}

\section{The algorithm}
In this section we prove our main theorem, and we present an algorithm for the permanent of a random tensor $ R \sim \mathscr{T}_{d,n,\mu} $ with $ \vert \mu \vert \geq (\frac{\ln n}{(d-1)^{6}})^{-c} $. We refer the interested reader to \cite{pin} and \cite{alg} for study about approximation algorithms and more information on complexity of algorithms.
\begin{theorem}\label{22}
For any constant $ c \in (0,\frac{1}{8}) $, there exists a deterministic quasi-polynomial time algorithm $ \mathscr{P} $, where the tensor $ R $ sampled from $ \mathscr{T}_{d,n,\mu} $ and $d \in \mathcal{O}(polylog(n))  $, for $\vert \mu \vert  \geq (\frac{\ln n}{(d-1)^{6}})^{-c} $, and a real number $\varepsilon \in (0,1)$, are considered as inputs. The algorithm computes a complex number $ \mathscr{P}(R,\varepsilon) $ that approximates the permanent of the tensor $R$ such that
\[
P\bigg( \bigg \vert 1-\dfrac{\mathscr{P}(R,\varepsilon)}{per(R)} \bigg\vert \leq \varepsilon \bigg)\geq 1-o(1),
\]
where the probability is taken over the random tensor $ R $.
\end{theorem}
\begin{proof}
We introduce a tensor $ X=J+zA $, where $ J $ is a tensor whose all entries are equal to 1, $ A $ is a random tensor with i.i.d entries sampled from $\mathcal{D}$, and $ z $ is a complex variable. We know that $ per(X) $ is a summation of $ (n!)^{d-1} $ products. Thus, we consider the normalized permanent $ \dfrac{per(X)}{(n!)^{d-1}} $. By Lemma \ref{11}, we know that the normalized permanent can be written as $ \sum\limits_{k=0}^{n} a_{k}z^{k} $, where
\begin{eqnarray*}
a_{k}=\dfrac{1}{(n^{\underline{k}})^{d-1}} \sum\limits_{B\subseteq_{k} A}per(B).
\end{eqnarray*}
To approximate $ \dfrac{per(X)}{(n!)^{d-1}} $, our algorithm computes the first $t+1$ coefficients $a_{0},a_{1},\dots,a_{t}$ and outputs the number $ \sum\limits_{k=0}^{t} a_{k}z^{k} $. For such a choice of $t$, the algorithm has time complexity $ t^{d-1} (\binom{n}{t})^{d} (t!)^{d-1}=\mathcal{O}(n^{dt}) $, which is quasi-polynomial. Now, we show that $ \sum\limits_{k=1}^{t}a_{k}z^{k} $ is a $\Theta(\varepsilon)$ approximation of $ \dfrac{per(X)}{(n!)^{d-1}}$ with high probability.

 By Lemma \ref{12}, 
\begin{eqnarray}\label{13}
\bigg \vert \sum\limits_{k=t+1}^{n} a_{k} z^{k} \bigg \vert \leq n^{-\gamma} \varepsilon,
\end{eqnarray}
with high probability (that is, this is small in the absolute sense). To show that relative error is small, we need to prove that
\begin{eqnarray}\label{14}
\bigg \vert \sum\limits_{k=0}^{t} a_{k}z^{k}\bigg \vert = \Omega(n^{-\gamma}),
\end{eqnarray}
with high probability. As the constant in $ \Theta(n^{-\gamma}) $ does not depend on $ \varepsilon $, relations (\ref{13}) and (\ref{14}) prove the theorem. To prove relation (\ref{14}), we follow the steps below.
By Lemma \ref{16}, with high probability we obtain
\[
 \bigg \vert \sum_{k=0}^{t} a_{k}z^{k} - \sum_{k=0}^{t} V_{k}z^{k} \bigg\vert \leq n^{-\beta} = o(n^{-\gamma}).
  \]
The difference between $n^{-\beta}$ and $ n^{-\gamma} $ is not important. 
\\
By Lemma \ref{42}, we know that 
\[
\bigg \vert \sum\limits_{k=0}^{t} V_{k}z^{k}-\sum\limits_{k=0}^{t} V_{k}^{\prime}z^{k}\bigg \vert = \mathscr{O}(n^{c-\nu})=o(n^{-\gamma}).
\]
In other words, we can estimate $ \sum\limits_{k=0}^{t} V_{k}z^{k} $ by $ \sum\limits_{k=0}^{t} V_{k}^{\prime}z^{k} $. Lemma \ref{17} allows us to conclude that we can estimate $\sum\limits_{k=0}^{t} V_{k}^{\prime}z^{k} $ by $ \sum\limits_{k=0}^{\infty} V_{k}^{\prime}z^{k} $. Now, it is enough to give an $ \Omega(n^{-\gamma}) $ lower bound for $ \bigg\vert \sum\limits_{k=0}^{\infty} V_{k}^{\prime}z^{k} \bigg\vert$.
It follows from \eqref{m33} that $ \sum\limits_{k=0}^{\infty} V_{k}^{\prime}z^{k} $ is simply $ e^{V_{1}z-\frac{\xi z^{2}}{2}}$, where $V_{1}$ is the normalized average of all entries of $A$, and $\vert V_{1}\vert$ is small with high probability by Lemma \ref{m31}. By Lemma \ref{18}, we conclude that with high probability,
\[
\bigg\vert e^{V_{1}z-\frac{\xi z^{2}}{2}} \bigg\vert \geq n^{-\gamma}.
\]
This completes the proof.
\end{proof}
Similar to the matrix case, we conclude from Theorem \ref{22} that, if we relax the approximation requirement, we can compute $V_{1}$ and estimate $ per(X) $ by $(n!)^{d-1} e^{V_{1}z-\frac{\xi z^{2}}{2}}$. Similar to the case of matrices, we achieve an approximation guarantee of the form below. 
\begin{corollary}
For any constant $ c \in (0,\frac{1}{8}) $ and $ 0< \rho<\frac{1}{8}-c $, there exists a deterministic polynomial time algorithm $ \mathcal{P} $ such that, given a tensor $R$ sampled from $ \mathscr{T}_{d,n,\mu} $, where $d \in \mathcal{O}(polylog(n))  $, for $\vert \mu \vert  \geq (\frac{\ln n}{(d-1)^{6}})^{-c} $, the algorithm computes a complex number $ \mathcal{P}(R) $ that approximates $ per(R) $ on average, in the sense that
\[
P\bigg( \bigg \vert 1-\dfrac{\mathcal{P}(R)}{per(R)} \bigg\vert \leq n^{-\rho} \bigg)= 1-o(1),
\]
where the probability is taken over the random tensor $ R $.
\end{corollary}
Similar to the matrix case, this algorithm provides a very good approximation for large $n$. But for small $n$, this is not a good algorithm. Also, we can convert this algorithm into a PTAS whose running time is polynomial in $n$, but possibly exponential in $ \frac{1}{\varepsilon}$. Similar to the matrix case, we fix a constant $0< \rho <\frac{1}{8}-c$ and for a given $\varepsilon$, if $\varepsilon > n^{-\rho}$ we use this algorithm, and otherwise we do not use it. Also, similar to the matrix case, we obtain the following corollary.
\begin{corollary}
For any constant $ c \in (0,\frac{1}{8}) $, there exists a deterministic PTAS algorithm that approximates $per(R)$ for a random tensor $R$ sampled from $ \mathscr{T}_{d,n,\mu} $, where $d \in \mathcal{O}(polylog(n))  $, for $\vert \mu \vert  \geq (\frac{\ln n}{(d-1)^{6}})^{-c} $.
\end{corollary} 


\section{Conclusion}
 In recent years, many researchers tried to approximate the permanents of matrices sampled from certain distributions \cite{abm}, \cite{birv2}, \cite{eld}. Having in mind the applications of the permanents of random matrices with zero or vanishing mean in quantum supremacy and Boson Sampling, Zhengfeng Ji, Zhihan Jin and Pinyan Lu presented an algorithm to approximate the permanent of a random matrix with module of mean at least equal to $\frac{1}{polylog(n)}$ \cite{pin}. In this paper, we extended this algorithm to tensors, and we designed a deterministic quasi-polynomial time algorithm and a PTAS to compute the permanent of a complex random tensor for which $d \in \mathcal{O}(polylog(n))  $, and module of mean was at least $\frac{1}{polylog(n)}$. 
 
 



\end{document}